\documentclass[a4paper,11pt]{article}
\usepackage{amssymb}
\usepackage{amsmath}
\usepackage{amsthm}
\allowdisplaybreaks
\usepackage{amscd}
\usepackage{graphicx}
\usepackage{hyperref}
\usepackage{geometry}
\usepackage{pxfonts}

\def\vbar{\mathchoice{\vrule height6.3ptdepth-.5ptwidth.8pt\kern- .8pt}
{\vrule height6.3ptdepth-.5ptwidth.8pt\kern-.8pt} {\vrule
height4.1ptdepth-.35ptwidth.6pt\kern-.6pt} {\vrule
height3.1ptdepth-.25ptwidth.5pt\kern-.5pt}}

\def\<{\langle}
\def\>{\rangle}
\def\a{\alpha}

\def\c{\cdot}

\def\nw{\nwarrow}
\def\ne{\nearrow}

\newtheorem{thm}{Theorem}[section]
\newtheorem{lem}[thm]{Lemma}
\newtheorem{cor}[thm]{Corollary}
\newtheorem{pro}[thm]{Proposition}

\newtheorem{rem}[thm]{Remark}

\theoremstyle{definition}
\newtheorem{defi}{Definition}[section]
\date{}
\begin{document}
\title{Nijenhuis Operators on $3$-Hom-L-dendriform algebras}
\author{A. Ben Hassine, T. Chtioui and  S. Mabrouk}
\author{\normalsize \bf  A. Ben Hassine\small{$^{1,2}$} \footnote { Corresponding author,  E-mail: benhassine.abdelkader@yahoo.fr}, T. Chtioui\small{$^{2}$} \footnote { Corresponding author,  E-mail: chtioui.taoufik@yahoo.fr},  S. Mabrouk\small{$^{3}$} \footnote { Corresponding author,  E-mail: mabrouksami00@yahoo.fr}}
\date{{\small{$^{1}$ Department of Mathematics, Faculty of Science and Arts at
Belqarn, P. O. Box 60, Sabt Al-Alaya 61985, University of Bisha, Kingdom of Saudi Arabia \\  \small{$^{2}$    Faculty of Sciences, University of Sfax,   BP
1171, 3000 Sfax, Tunisia \\  \small{$^{3}$} Faculty of Sciences, University of Gafsa,   BP
2100, Gafsa, Tunisia
 }}}} \maketitle

\maketitle
\begin{abstract}
The goal of this work is to introduce the notion of $3$-Hom-Lie-dendriform algebras which is the dendriform version of $3$-Hom-Lie-algebras. They can be also regarded as the ternary analogous of Hom-Lie-dendriform algebras. We give the representation  of a  $3$-Hom-pre-Lie algebra. Moreover, we introduce the notion of Nijenhuis operators on a $3$-Hom-pre-Lie algebra and provide some constructions of $3$-Hom-Lie-dendriform algebras in term of  Nijenhuis operators. Parallelly, we introduce the notion of a product and complex structures
on a $3$-Hom-Lie-dendriform algebras and there are also four types special integrability conditions.
\end{abstract}
%\numberwithin{equation}{section}

\tableofcontents
\section*{Introduction}
Hom-Lie algebras were introduced in the context of q-deformation of Witt and Virasoro algebras. In a
sequel, various concepts and properties have been derived to the framework of other hom-algebras as well.
The study of Hom-algebras appeared extensively in the work of J. Hartwig, D. Larsson, A. Makhlouf, S. Silvestrov, D. Yau and other authors (\cite{Ammar&Ejbehi&Makhlouf,Hartwig&Larsson&Silvestrov,Makhlouf&Silvestrov2008,Makhlouf&Silvestrov2010,Yau2009}).
 Other types of algebras (such as associative, Leibniz,
Poisson, Hopf,. . .) twisted by homomorphisms have also been studied in the last few years.\\
Recently, there have been several interesting developments of Hom-Lie algebras in mathematics and mathematical physics, including Hom-Lie bialgebras \cite{chen&Wang&Zhang,Chen&Zheng&Zhang}, quadratic Hom-Lie
algebras \cite{Benayadi&Makhlouf}, involutive Hom-semigroups \cite{Zheng&Guo}, deformed vector fields and differential calculus \cite{Larsson&Silvestrov}, representations \cite{Sheng,{Zhao&Chen&Ma}}, cohomology and homology theory \cite{Ammar&Ejbehi&Makhlouf,Yau2009}, Yetter-Drinfeld
categories \cite{Wang&Wang}, Hom-Yang-Baxter equations \cite{Chen&Zhang,ChenZhang2018,Liu&Makhlouf&Menini&Panaite,Sheng,Yau2011}, Hom-Lie 2-algebras \cite{Sheng&Chen,Song&Tang},
$(m,n)-$Hom-Lie algebras \cite{Ma&Zheng}, Hom-left-symmetric algebras \cite{Makhlouf&Silvestrov2008} and enveloping algebras \cite{Guo&Zhang&Zheng}.\\
The twisted version of another type algebras, called dendriform algebras. These algebras were introduced by Loday as Koszul dual of associative dialgebras \cite{Loday2001}. Free dendriform algebra over a vector space has been constructed by using planar binary trees. Dendriform algebras also arise from Rota-Baxter operators on some associative algebra \cite{Aguiar200}. Recently, an
explicit cohomology theory for dendriform algebras has been introduced and the formal deformation
theory for dendriform algebras (as well as coalgebras) has been studied in \cite{Das1,Das2}. The cohomology
involves certain combinatorial maps. Some results in 3-L-dendriform algebras are given in \cite{Chtioui&Mabrouk}.

The twisted version of dendriform structures, called Hom-dendriform structures was introduced
in \cite{Makhlouf2012}. These algebras can be thought of as splitting of Hom-associative algebras. They also arise
from Rota-Baxter operator on Hom-associative algebras. In \cite{Ma&Zheng} the authors study Hom-dendriform
algebras from the point of view of monoidal categories. In \cite{Das2020} is given the twist of the construction of \cite{Das1} by a homomorphism to formulate a cohomology for Hom-dendriform algebras.

Deformations of n-Lie algebras have been studied from several aspects. See \cite{Azcrraga&Izquierdo,Figueroa}
for more details. In particular, a notion of a Nijenhuis operator on a 3-Lie algebra was
introduced in \cite{Zhang} in the study of the 1-order deformations of a 3-Lie algebra. But there
are some quite strong conditions in this definition of a Nijenhuis operator. In the case of
Lie algebras, one could obtain fruitful results by considering one-parameter infinitesimal
deformations, i.e. 1-order deformations. However, for n-Lie algebras, we believe that
one should consider (n - 1)-order deformations to obtain similar results. In \cite{Figueroa}, for $3-$
Lie algebras, the author had already considered 2-order deformations. For the case of pre-Lie algebras, the authors in \cite{Wang&Sheng&Bai&Liu} give the notion of Nijenhuis operator.\\
Thus it is time to study $3$-Hom-Lie-dendriform algebras, $3$-Hom-pre-Lie algebras and Nijenhuis operator. Similarly, we give the some proprieties of Nijenhuis operator on $3$-Hom-Lie-dendriform algebras and the relationship between the Nijenhuis operator on $3$-Hom-Lie-dendriform algebras and the Nijenhuis operator on $3$-Hom-pre-Lie algebras.

Product structures and complex structures on a $3$-Hom-Lie algebra
can be viewed as special Nijenhuis operator(\cite{ShengTang}).  These structures were considered by many authors from different points of view.  For example, A. Andrada and Aderi\'an studied Complexproduct structures on $6$-dimensional nilpotent Lie algebras(\cite{Andrada3}), A. Andrada, M. L. Barberis andI. Dotti studied Classification of abelian complex structures on $6$-dimensional Lie algebras(\cite{Andrada1}).A. Andrada, M. L. Barberis and I. Dotti studied Product structures on four dimensional solvable Lie algebras(\cite{Andrada2}).
\\
This paper is organized as follows: In Section 1, we recall the concepts of 3-Hom-Lie algebra and introduce the notion of $3-$Hom-pre-Lie algebra, representation of $3-$Hom-pre-Lie algebra. Some results and the definition of $3-$Hom-L-dendriform  are established
in Section 2. Section 3 is  dedicated to study the second order deformation of $3$-Hom-Lie-dendriform and introduce the
notion of Nijenhuis operator on $3-$Hom-Lie-dendriform, which could generate a trivial deformation. In the other part of this section we give some properties and results of Nijenhuis operators. Finally, we introduce the notion of
product and complex structures on a $3$-Hom-L-dendriform algebra and prove that a
it is the direct sum of two Hom-subalgebras (as vector spaces) if and only if there is a
product (resp. complex) structure on it. Next, we study several special product and
complex structures, which are called
strict product (complex) structures, abelian product (complex) structures,
strong abelian product (complex) structures and
perfect product (complex) structures, respectively.

All  vector spaces are considered over field $\mathbb{K}$ of characteristic $0$. Moreover  all ternary Hom-algebras are supposed multiplicative, that is the twisted map preserved ternary operations.

%%%%%%%%%%%%%%%%%%%%%%%%%%%%%%%%%%%%%%
\section{Preliminaries and basics}
%%%%%%%%%%%%%%%%%%%%%%%%%%%%%%%%%%%%%%
\label{sec:bas}

In this section, we give  some general results about $3$-Hom-Lie algebras   (see \cite{Ataguema&Makhlouf&Silvestrov}). First, recall that a {$3$-Hom-Lie algebra}  is a vector space $A$ together with a skew-symmetric linear map ($3$-Hom-Lie bracket) $[\cdot,\cdot,\cdot]:
\otimes^3 A\rightarrow A$ and a linear map $\alpha:A\rightarrow A$ such that the following fundamental identity (FI) holds:
 \begin{equation}\label{eq:de1}
 \begin{split}
&[\alpha(x_1),\alpha(x_2),[x_3,x_4,x_5]]=[[x_1,x_2,x_3],\alpha(x_4),\alpha(x_5)]\\
&+[\alpha(x_3),[x_1,x_2,x_4],\alpha(x_5)]+[\alpha(x_3),\alpha(x_4),[x_1,x_2,x_5]]
\end{split}
\end{equation}
for $x_i\in A, 1\leq i\leq 5$.
If $\a [\c,\c,\c]=[\c,\c,\c]\circ \a^{\otimes^3}$, we say that $A$ is multiplicative.
Define for any $x_1, x_2\in A$,  the linear map
\begin{equation}\label{eq:adjoint}
ad:A\wedge A\to gl(A), \quad ad(x_1,x_2)(x):=[x_1,x_2,x], \quad \forall x\in A.
\end{equation}
The identity \label{eq:de1} can be written as follows:
  \begin{align*}
ad(\alpha(x_1),\alpha(x_2))[x_3,x_4,x_5]=&[ad(x_1,x_2)(x_3),\alpha(x_4),\alpha(x_5)] +[\alpha(x_3),(x_1,x_2)(x_4),\alpha(x_5)] \\
+&[\alpha(x_3),\alpha(x_4),ad(x_1,x_2)(x_5)],
\end{align*}
for  all $x_1, x_2, x_3\in A$. A morphism of $3$-Hom-Lie algebras $\phi:(A_1,[\cdot,\cdot,\cdot]_{A_1},\alpha_1)\rightarrow(A_2,[\cdot,\cdot,\cdot]_{A_2},\alpha_2)$ is a linear map such that for all $x_1,x_2,x_3\in A$,
\begin{align*}
\phi([x_1,x_2,x_3]_{A_1})=&[\phi(x_1),\phi(x_2),\phi(x_3)]_{A_2},\\
\phi\circ\alpha_1=&\alpha_2\circ\phi.
\end{align*}
%\begin{ex}\cite{Filippov} Consider $4$-dimensional $3$-Lie algebra $A$ generated by $(e_1,e_2,e_3,e_4)$  with  the following multiplication
%$$\left\{
%    \begin{array}{ll}
%      $[$e_1,e_2,e_3$]$=e_4,& \hbox{} \\
%$[$e_1,e_2,e_4$]$=e_3,& \hbox{} \\
%$[$e_1,e_3,e_4$]$=e_2,& \hbox{} \\
%     $[$ e_2,e_3,e_4$]$=e_1 .& \hbox{}
%    \end{array}
%  \right.
%$$
%
%\end{ex}
%\begin{pro}\label{pro:someequalities}
%Let $A$ be a vector space  together with a skew-symmetric linear map $[\cdot,\cdot,\cdot]:
%\otimes^3 A\rightarrow A$ and $\alpha:A\rightarrow A$ a linear map. Then $(A,[\cdot,\cdot,\cdot],\alpha)$ is a $3$-Hom-Lie algebra if and only if the following identities hold:
%\begin{enumerate}
%\item {\small  $[[x_1,x_2,x_3],\alpha(x_4),\alpha(x_5)]-[[x_1,x_2,x_4],\alpha(x_3),\alpha(x_5)]+[[x_1,x_3,x_4],\alpha(x_2),\alpha(x_5)]-[[x_2,x_3,x_4],\alpha(x_1),\alpha(x_5)]=0$,}
%\item $[[x_1,x_2,x_5],\alpha(x_3),\alpha(x_4)]+[[x_3,x_4,x_5],\alpha(x_1),\alpha(x_2)]-[[x_1,x_3,x_5],\alpha(x_2),\alpha(x_4)]
%\\-[[x_2,x_4,x_5], \alpha(x_1),\alpha(x_3)]+[[x_1,x_4,x_5],\alpha(x_2),\alpha(x_3)]+[[x_2,x_3,x_5],\alpha(x_1),\alpha(x_4)]=0$,
%\end{enumerate}for  $x_i\in A, 1\leq i\leq 5$.\label{pro:new}
%\end{pro}
The notion of a representation of an $n$-Hom-Lie algebra was introduced in \cite{Ammar&mabrouk&makhlouf}.
\begin{defi}\label{defi:rep}
 Let $(A,[\cdot,\cdot,\cdot],\alpha)$ be a $3$-Hom-Lie algebra, $V$ be a vector space and $\phi:V\to V$ be a linear map. A representation of $A$ on $V$ with respect to $\phi$ is a  linear map $\rho: \wedge^2A\rightarrow gl(V)$ such that
\begin{eqnarray}
&&\phi\rho(x_1,x_2)=\rho(\alpha(x_1),\a(x_2))\phi,\\
&&\nonumber\rho (\alpha(x_1),\alpha(x_2))\rho(x_3,x_4)-\rho(\alpha(x_3),\alpha(x_4))\rho(x_1,x_2)\\
 &&\label{Rep3HomLie1}  =\rho([x_1,x_2,x_3],\alpha(x_4))\circ\phi-\rho([x_1,x_2,x_4],\alpha(x_3))\circ\phi,\\
&&\nonumber\rho ([x_1,x_2,x_3],\alpha(x_4))\circ\phi=\rho(\alpha(x_1),\alpha(x_2))\rho(x_3,x_4)\\
&&\label{Rep3HomLie2} +\rho(\alpha(x_2),\alpha(x_3))\rho(x_1,x_4)+\rho(\alpha(x_3),\alpha(x_1))\rho(x_2,x_4)
\end{eqnarray}
for $x_i\in A, 1\leq i\leq 4$. We say that $(V,\rho,\phi)$ is a representation of $A$ or $V$ is an $A$-module.
\end{defi}
The tuple $(A,ad,\a)$ is a representation of $A$ on itself which is called the adjoint  representation.

\begin{lem}
Let $(A,[\cdot,\cdot,\cdot],\alpha)$ be a $3$-Hom-Lie algebra, $V$  a vector space , $\phi\in End(V)$ and $\rho:
\wedge^2A\rightarrow  gl(V)$  a  linear
map. Then $(V,\rho,\phi)$ is a representation of $A$ if and only if there
is a $3$-Hom-Lie algebra structure $($called the semi-direct product$)$
on the direct sum $A\oplus V$ of vector spaces, defined by
\begin{align}
[x_1+v_1,x_2+v_2,x_3+v_3]_{A\oplus V}=&[x_1,x_2,x_3]+\rho(x_1,x_2)v_3+\rho(x_3,x_1)v_2+\rho(x_2,x_3)v_1, \\
\alpha_{A\oplus V}(x_1+v_1)=&\alpha(x_1)+\phi(v_1),
\end{align}
for $x_i\in A, v_i\in V, 1\leq i\leq 3$. We denote this semi-direct product  $3$-Hom-Lie algebra by $A\ltimes_\rho V.$
\end{lem}

\begin{pro}\label{pro:rep properties}
Let $(V,\rho,\phi)$ be a representation of a $3$-Hom-Lie algebra $A$. Then the following identities hold:
{\small\begin{align}\label{pro:new-repn}
&(\rho([x_1,x_2,x_3],\alpha(x_4))-\rho([x_1,x_2,x_4],\alpha(x_3)])\nonumber\\&+\rho([x_1,x_3,x_4],\alpha(x_2))
    -\rho([x_2,x_3,x_4],\alpha(x_1)))\circ\phi=0,\\
&\rho(\alpha(x_1),\alpha(x_2))\rho(x_3,x_4)+\rho(\alpha(x_2),\alpha(x_3))\rho(x_1,x_4)+\rho(\alpha(x_3),\alpha(x_1))\rho(x_2,x_4)
\nonumber\\& +\rho(\alpha(x_3),\alpha(x_4))\rho(x_1,x_2)+\rho(\alpha(x_1),\alpha(x_4))\rho(x_2,x_3)+\rho(\alpha(x_2),\alpha(x_4))\rho(x_3,x_1)=0
\end{align}}
for $x_i\in A, 1\leq i\leq 4$.
\end{pro}

Let $(V,\rho,\phi)$ be a representation of a $3$-Hom-Lie algebra $(A,[\cdot,\cdot,\cdot],\a)$. In the sequel, we always assume that $\phi$ is invertible. Define $\rho^*: A\wedge A \longrightarrow gl(V^*)$ as usual by
$$\langle \rho^*(x,y)(\xi),u\rangle=-\langle\xi,\rho(x,y)(u)\rangle,\quad\forall x,y\in A,u\in V,\xi\in V^*.$$
However, in general $\rho^*$ is not a representation of $A$ anymore. To our knowledge, people need to add a very strong condition to obtain a representation on the dual space in the former study.  We recall the dual representation of a representation of a $3$-Hom-Lie algebra without any additional condition.

Define $\rho^\star:A \wedge A \longrightarrow gl(V^*)$ by
\begin{equation}\label{eq:new1}
 \rho^\star(x,y)(\xi):=\rho^*(\a(x),\a(y))\big{(}(\phi^{-2})^*(\xi)\big{)},\quad\forall x,y\in A,\xi\in V^*.
\end{equation}
More precisely, we have
\begin{eqnarray}\label{eq:new1gen}
\langle\rho^\star(x,y)(\xi),u\rangle=-\langle\xi,\rho(\a^{-1}(x),\a^{-1}(y))(\phi^{-2}(u))\rangle,\quad\forall x,y\in A, u\in V, \xi \in V^*.
\end{eqnarray}
\begin{lem}\label{lem:dualrep}
 Let $(V,\rho,\phi )$ be a representation of a 3-Hom-Lie algebra $(A,[\cdot,\cdot,\cdot],\a)$. Then $\rho^\star:A \longrightarrow gl(V^*)$ defined above by \eqref{eq:new1} is a representation of $(A,[\cdot,\cdot,\cdot],\a)$ on $V^*$ with respect to $(\phi^{-1})^*$.
\end{lem}
\begin{proof}
For any $x,y \in A, \xi \in V^*$ and $u \in V$, we have
\begin{align*}
\langle \rho^\star(\a(x),\a(y))((\phi^{-1})^*(\xi)),u \rangle=&
-\langle (\phi^{-1})^*(\xi),\rho(x,y)\phi^{-2}(u) \rangle \\
=&-\langle \xi,\rho(\a^{-1}(x),\a^{-1}(y))\phi^{-3}(u)\rangle
\end{align*}
and
\begin{align*}
\langle (\phi^{-1})^* \rho^\star(x,y)(\xi),u\rangle=&
\langle  \rho^\star(x,y)(\xi),\phi^{-1}(u) \rangle \\
=&-\langle \xi,\rho(\a^{-1}(x),\a^{-1}(y))\phi^{-3}(u)\rangle,
\end{align*}
which implies that $(\phi^{-1})^*\rho^\star(x,y)=\rho^\star(\alpha(x),\a(y))(\phi^{-1})^*$.
On the other hand,
\begin{align*}
&\langle  \rho^\star(\a(x_1),\a(x_2))\rho^\star(x_3,x_4) (\xi),u\rangle
-\langle\rho^\star(\a(x_3),\a(x_4))\rho^\star(x_1,x_2) (\xi),u\rangle \\
&=\langle \xi,\rho(\a^{-1}(x_3),\a^{-1}(x_4))\rho(\a^{-2}(x_1),\a^{-2}(x_2))\phi^{-4}(u)\rangle \\
&-   \langle \xi,\rho(\a^{-1}(x_1),\a^{-1}(x_2))\rho(\a^{-2}(x_3),\a^{-2}(x_4))\phi^{-4}(u)\rangle,
\end{align*}
\begin{align*}
&\langle \rho^\star([x_1,x_2,x_3],\a(x_4))(\phi^{-1})^* (\xi),u \rangle   -
\langle \rho^\star([x_1,x_2,x_4],\a(x_3))(\phi^{-1})^*(\xi),u \rangle \\
&=-\langle \xi,\rho([\a^{-2}(x_1),\a^{-2}(x_2),\a^{-2}(x_3)],\a^{-1}(x_4))\phi^{-3}(u) \rangle \\
&+\langle \xi, \rho([\a^{-2}(x_1),\a^{-2}(x_2),\a^{-2}(x_4)],\a^{-1}(x_3))\phi^{-3}(u) \rangle.
\end{align*}
Since $(\rho,\phi)$ is a representation, then we obtain
\begin{align*}
& \rho^\star(\a(x_1),\a(x_2))\rho^\star(x_3,x_4)- \rho^\star(\a(x_3),\a(x_4))\rho^\star(x_1,x_2) \\
=&\rho^\star([x_1,x_2,x_3],\a(x_4))(\phi^{-1})^* -
\rho^\star([x_1,x_2,x_4],\a(x_3))(\phi^{-1})^*.
\end{align*}
The identity \eqref{Rep3HomLie2} can be shown similarly. Therefore $\rho^\star$  is a representation of $A$ on $V^*$ with respect to $(\phi^{-1})^*$.
\end{proof}

The tuple $(V^*,\rho^\star, (\phi^{-1})^*)$ is called the dual representation of $(V,\rho,\phi)$ for the $3$-Hom-Lie algebra  $(A,[\cdot,\cdot,\cdot],\a)$.

\begin{cor}
Let $(A,[\c,\c,\c],\a)$ be a $3$-Hom-Lie algebra. Then $ad^\star: A \to gl(A^*)$ defined by
\begin{align}\label{coadjoint rep}
\langle ad^\star(x,y)(\xi),z\rangle=-\langle\xi,ad(\a^{-1}(x),\a^{-1}(y))(\a^{-2}(z))\rangle,\quad\forall x,y,z\in A, \xi \in A^*
\end{align}
is a representation of  $(A,[\c,\c,\c],\a)$ on $A^*$ with respect to $(\a^{-1})^*$, which is called the coadjoint representation.
\end{cor}

In the next section, we recall the definition of $3$-Hom-pre-Lie algebras and give construction result in terms of  {$\mathcal O$-operators}  on $3$-Hom-Lie algebras $($see \cite{Guo&Zhang&Wang}$)$. The notion of $\mathcal{O}$-operator is useful tool for the construction of the solutions of the classical Yang-Baxter equation.
Let $(A,[\cdot,\cdot,\cdot],\alpha)$ be a $3$-Hom-Lie algebra and $(V,\rho,\phi)$
a representation.  A linear operator $T:V\rightarrow A$ is called
an $\mathcal O$-operator associated to $( V,\rho,\phi)$ if $T$
satisfies
\begin{align}\label{eq:Ooperator}
 [Tu,Tv,Tw]=&T\big(\rho(Tu,Tv)w+\rho(Tv,Tw)u+\rho(Tw,Tu)v\big),\quad \forall u,v,w\in V, \\
 T \phi=& \alpha  T.
\end{align}
A linear map $\mathcal R:A\to A$ is called a Rota-Baxter operator if its an $\mathcal O$-operator associated to the adjoint representation. That is $\mathcal R\a=\a R$ and for any $x,y,z\in A$ we have
$$[\mathcal R(x),\mathcal R(y),\mathcal R(z)]=\mathcal R([\mathcal R(x),\mathcal R(y),z]+[\mathcal R(x),y,\mathcal R(z)]+[x,\mathcal R(y),\mathcal R(z)]).$$

%%%%%%%%%%%%%%%%%%%%%%%%%%%%%%%%%%%%%%%%
\section{Representations of $3$-Hom-pre-Lie algebras}
%%%%%%%%%%%%%%%%%%%%%%%%%%%%%%%%%%%%%%%%
In this section, we introduce the notion of a representation of a $3$-Hom-pre-Lie algebra, construct the corresponding semidirect product $3$-Hom-Lie algebra and give the dual representation  of a  given representation without any additional condition.
\begin{defi}(\cite{Guo&Zhang&Wang})
Let $A$ be a vector space with a linear map $\{\cdot,\cdot,\cdot\}:A\otimes A\otimes A\rightarrow A$ and $\alpha:A\rightarrow A$ is a linear map.
The triplet $(A,\{\cdot,\cdot,\cdot\},\alpha)$ is called a $3$-Hom-pre-Lie algebra if the following identities hold:
\begin{eqnarray}
\{x,y,z\}&=&-\{y,x,z\},\label{3-pre-Lie 0}\\
\nonumber\{\alpha(x_1),\alpha(x_2),\{x_3,x_4,x_5\}\}&=&\{[x_1,x_2,x_3]^C,\alpha(x_4),\alpha(x_5)\}+\{\alpha(x_3),[x_1,x_2,x_4]^C,\alpha(x_5)\}\\
&&+\{\alpha(x_3),\alpha(x_4),\{x_1,x_2,x_5\}\},\label{3-hom-pre-Lie 1}\\
\nonumber\{ [x_1,x_2,x_3]^C,\alpha(x_4), \alpha(x_5)\}&=&\{\alpha(x_1),\alpha(x_2),\{ x_3,x_4, x_5\}\}+\{\alpha(x_2),\alpha(x_3),\{ x_1,x_4,x_5\}\}\\
&&+\{\alpha(x_3),\alpha(x_1),\{ x_2,x_4, x_5\}\},\label{3-hom-pre-Lie 2}
\end{eqnarray}
 where $x,y,z, x_i\in A, 1\leq i\leq 5$ and $[\cdot,\cdot,\cdot]^C$ is defined by
\begin{equation}
[x,y,z]^C=\{x,y,z\}+\{y,z,x\}+\{z,x,y\},\quad \forall  x,y,z\in A.\label{eq:3cc}
\end{equation}
\end{defi}

 \begin{pro}(\cite{Guo&Zhang&Wang})
Let $(A,\{\cdot,\cdot,\cdot\},\alpha)$ be a $3$-Hom-pre-Lie algebra. Then the induced $3$-commutator given by Eq.~\eqref{eq:3cc} defines
a $3$-Hom-Lie algebra.
It  is called the  sub-adjacent $3$-Hom-Lie algebra of $(A,\{\cdot,\cdot,\cdot\},\alpha)$ and $(A,\{\cdot,\cdot,\cdot\},\alpha)$ is the compatible
$3$-Hom-pre-Lie algebra of  $(A,[\cdot,\cdot,\cdot]^C,\alpha)$.\\
The sub-adjacent $3$-Hom-Lie algebra $(A,\{\cdot,\cdot,\cdot\},\alpha)$ will be denoted by $A^c$.
\end{pro}
Let $(A,\{\cdot,\cdot,\cdot\},\alpha)$ be a $3$-Hom-pre-Lie algebra. Define the left multiplication $L:\wedge^2 A\longrightarrow gl(A)$ by
$L(x,y)z=\{x,y,z\}$ for all $x,y,z\in A$. Then $(A,L,\alpha)$ is a representation of the
$3$-Hom-Lie algebra $A^c$. Moreover, we define the right multiplication
$R:\otimes^2 A \to  gl(A)$ by $R(x,y)z=\{z,x,y\}$.

By the definition of  $3$-Hom-pre-Lie algebra and  the representation of a $3$-Hom-Lie algebra, we immediately obtain
\begin{pro}
If $A$ is  a vector space with linear maps
$\{\cdot,\cdot,\cdot\}:A\otimes A\otimes A\rightarrow A$ and $\alpha:A\rightarrow A$
satisfying
 Eq.~\eqref{3-hom-pre-Lie 1}. Then $(A,\{\cdot,\cdot,\cdot\},\alpha) $ is a $3$-Hom-pre-Lie algebra if $(A,[\cdot,\cdot,\cdot]^C,\alpha)$ is a $3$-Hom-Lie algebra and $(A,L,\alpha)$ gives a representation of it.
\end{pro}

New identities of 3-Hom-pre-Lie algebras can be derived from Proposition~\ref{pro:rep properties}. For example,

\begin{cor}
Let $(A,\{\cdot,\cdot,\cdot\},\alpha)$ be a $3$-Hom-pre-Lie algebra. Then the following identities hold:
{\small\begin{align}
&\nonumber\{ [x_1,x_2,x_3]^C,\alpha(x_4), \alpha(x_5)\}-\{[x_1,x_2,x_4]^C,\alpha(x_3), \alpha(x_5) \}\\
+&\{[x_1,x_3,x_4]^C,\alpha(x_2),\alpha( x_5)\}
-\{[x_2,x_3,x_4]^C,\alpha(x_1), \alpha(x_5)\}=0,\\
&\{\alpha(x_1),\alpha(x_2),\{x_3,x_4, x_5\}\}+\{\alpha(x_3),\alpha(x_4),\{x_1,x_2, x_5\}\}+\{\alpha(x_2),\alpha(x_4),\{x_3,x_1, x_5\}\}\nonumber\\
+&\{\alpha(x_3),\alpha(x_1),\{x_2,x_4, x_5\}\}+\{\alpha(x_2),\alpha(x_3),\{x_1,x_4, x_5\}\}+\{\alpha(x_1),\alpha(x_4),\{x_2,x_3, x_5\}\}=0,
\end{align}}
for $x_i\in A, 1\leq i\leq 5$.
\end{cor}

\begin{pro}(\cite{Guo&Zhang&Wang})\label{pro:3hompreLieT}
Let $(A,[\cdot,\cdot,\cdot],\alpha)$ be a $3$-Hom-Lie algebra and $(V,\rho,\phi)$ a representation. Suppose that the linear map $T:V\rightarrow A$ is an $\mathcal O$-operator associated
to $(V,\rho,\phi)$. Then $(V,\{\cdot,\cdot,\cdot\},\phi)$ is a $3$-Hom-pre-Lie algebra where the bracket $\{\cdot,\cdot,\cdot\}$ is given by
\begin{equation}
\{u,v,w\}=\rho(Tu,Tv)w,\quad\forall ~ u,v,w\in V.
\end{equation}
\end{pro}

\begin{cor}
With the above conditions,  $(V,[\cdot,\cdot,\cdot]^C,\phi)$ is a $3$-Hom-Lie
algebra as the sub-adjacent $3$-Hom-Lie algebra of the $3$-Hom-pre-Lie
algebra given in Proposition \ref{pro:3hompreLieT}, and $T$ is a $3$-Hom-Lie algebra morphism from $(V,[\cdot,\cdot,\cdot]^C,\phi)$ to $(A,[\cdot,\cdot,\cdot],\alpha)$. Furthermore,
$T(V)=\{Tv|v\in V\}\subset A$ is a $3$-Hom-Lie subalgebra of $A$ and there is an induced $3$-Hom-pre-Lie algebra structure $\{\cdot,\cdot,\cdot\}_{T(V)}$ on
$T(V)$ given by
\begin{equation}
\{Tu,Tv,Tw\}_{T(V)}:=T\{u,v,w\},\quad\;\forall u,v,w\in V.
\end{equation}
\end{cor}

\begin{cor}\label{pro:preLieOoper}
Let $(A,[\cdot,\cdot,\cdot],\alpha)$ be a $3$-Hom-Lie algebra. Then there exists a compatible $3$-Hom-pre-Lie algebra if and only
if there exists an invertible $\mathcal O$-operator on $A$.
\end{cor}

 Now, we introduce the definition of representation of  $3$-Hom-pre-Lie algebra and give some related results.
\begin{defi}\label{def-rep-3-hompre}
 A  representation of a $3$-Hom-pre-Lie algebra $(A,\{\cdot,\cdot,\cdot\},\alpha)$   on a vector space $V$ consists of a triplet $(l,r,\phi)$,
  where $l:\wedge^2 A \rightarrow gl(V)$ is a representation of the sub-adjacent $3$-Hom-Lie algebra $A^c$ on $V$ with respect to $\phi \in End(V)$ and $r:\otimes^2 A \rightarrow gl(V)$  is a linear map such that  for all $x_1,x_2,x_3,x_4\in A$, the following conditions hold
{\small \begin{align}
&\phi\circ r(x_1,x_2)=r(\alpha(x_1),\alpha(x_2))\circ\phi,\\
 &\label{rep1} l(\alpha(x_1),\alpha(x_2))r(x_3,x_4) =r(\alpha(x_3),\alpha(x_4))\mu(x_1,x_2)+r([x_1,x_2,x_3]^C,\alpha(x_4))\circ\phi+ r(\alpha(x_3),\{x_1,x_2,x_4\})\circ\phi, \\
&\label{rep2} r([x_1,x_2,x_3]^C,\alpha(x_4))\circ\phi=l(\alpha(x_1),\alpha(x_2))r(x_3,x_4)+l(\alpha(x_2),\alpha(x_3))r(x_1,x_4)+l(\alpha(x_3),\alpha(x_1))r(x_2,x_4),\\
 & r(\alpha(x_1),\{x_2,x_3,x_4\})\circ\phi =r(\alpha(x_3),\alpha(x_4))\mu(x_1,x_2) -r(\alpha(x_2),\alpha(x_4))\mu(x_1,x_3)
                         + l(\alpha(x_2),\alpha(x_3))r(x_1,x_4),  \label{rep3}\\
& r(\alpha(x_3),\alpha(x_4))\mu(x_1,x_2) =
                       l(\alpha(x_1),\alpha(x_2))r(x_3,x_4)           -r(\alpha(x_2),\{x_1,x_3,x_4\})\circ\phi+
                        r(\alpha(x_1),\{x_2,x_3,x_4\})\circ\phi,  \label{rep4}
\end{align}}
where $\mu(x,y)=l(x,y)+r(x,y)-r(y,x)$, for any $x,y \in A$.
\end{defi}
It is obvious that $(A,L,R,\alpha)$ is a
representation of a $3$-Hom-pre-Lie algebra on itself, which is called the adjoint
representation.
\begin{thm}
Let $(A,\{\cdot,\cdot,\cdot\})$ be 3-pre-Lie algebras and $(V,l,r)$ be a representation on $A$. Let  $\phi:V\rightarrow V$ and $\alpha:A\rightarrow A$ be two morphisms such that:
$$
\phi\circ l(x_1,x_2)=l(\alpha(x_1),\alpha(x_2))\circ\phi, ~~~\phi\circ r(x_1,x_2)=r(\alpha(x_1),\alpha(x_2))\circ\phi, \forall x_1,x_2\in A.
$$
Then $(V,\tilde l,\tilde r,\phi)$ is a representation of 3-Hom-pre-Lie algebras  $(A,\{\cdot,\cdot,\cdot\}_{\alpha,C}=\alpha\circ\{\cdot,\cdot,\cdot\}^C,\alpha)$, where $\tilde l=\phi\circ l$ and $\tilde r=\phi\circ r$.

\end{thm}
\begin{proof}
We will proof the equality \eqref{rep2}, i.e. for all $x_1,x_2,x_3,x_4\in A$
\begin{align*}
&\tilde r([x_1,x_2,x_3]_{\alpha,C},\alpha(x_4))\circ\phi-\tilde l(\alpha(x_1),\alpha(x_2))\tilde r(x_3,x_4) -\tilde l(\alpha(x_2),\alpha(x_3))\tilde r(x_1,x_4)-\tilde l(\alpha(x_3),\alpha(x_1))\tilde r(x_2,x_4)=0.
\end{align*}
We have
\begin{align*}
&\tilde r([x_1,x_2,x_3]_{\a,C},\a(x_4))\circ\phi-\tilde l(\a(x_1),\alpha(x_2))\tilde r(x_3,x_4) -\tilde l(\alpha(x_2),\a(x_3))\tilde r(x_1,x_4) -\tilde l(\a(x_3),\alpha(x_1))\tilde r(x_2,x_4)\\
=&\phi\circ r(\alpha\circ[x_1,x_2,x_3],\alpha(x_4))\circ\phi-\phi\circ l(\alpha(x_1),\alpha(x_2))\phi\circ r(x_3,x_4) \\
-&\phi\circ l(\alpha(x_2),\alpha(x_3))\phi\circ r(x_1,x_4)
-\phi\circ l(\alpha(x_3),\alpha(x_1))\phi\circ r(x_2,x_4)\\
=&\phi^2\circ \big(r([x_1,x_2,x_3],x_4)-l(x_1,x_2) r(x_3,x_4)- l(x_2,x_3) r(x_1,x_4)-l(x_3,x_1) r(x_2,x_4)\big)=0.
\end{align*}
The other identities of the definition \eqref{def-rep-3-hompre} can be shown similarly.
\end{proof}

\begin{pro}
Let $(A,\{\c,\c,\c\},\alpha)$ be a $3$-Hom-pre-Lie algebra, $V$  a vector space and $l,r:
\otimes^2A\rightarrow  gl(V)$  two  linear
maps and $\phi:V\rightarrow V$ a  morphism algebra. Then $(V,l,r,\phi)$ is a representation of $A$ if and only if there
is a $3$-Hom-pre-Lie algebra structure $($called the semi-direct product$)$
on the direct sum $A\oplus V$ of vector spaces, defined by
\begin{align}\label{eq:sum}
(\alpha+\phi)_{{A\oplus V}}(x_1+u_1)=&\alpha(x_1)+\phi(u_1),\\
[x_1+u_1,x_2+u_2,x_3+u_3]_{A\oplus V}=&\{x_1,x_2,x_3\}+l(x_1,x_2)u_3-r(x_1,x_3)u_2+r(x_2,x_3)u_1,
\end{align}
for $x_i\in A, u_i\in V, 1\leq i\leq 3$. We denote this semi-direct product $3$-Hom-pre-Lie algebra by $A\ltimes_{l,r}^{\a,\phi} V.$
\end{pro}

\begin{proof}
Let $x_i \in A$ and $u_i \in V$, for $1 \leq i \leq 5$. Then
{\small\begin{align*}
    &[(\alpha+\phi)_{{A\oplus V}}(x_1+u_1),(\alpha+\phi)_{{A\oplus V}}(x_2+u_2),[x_3+u_3,x_4+u_4,x_5+u_5]_{A\oplus V}]_{A\oplus V}\\
    -&
    [(\alpha+\phi)_{{A\oplus V}}(x_3+u_3),(\alpha+\phi)_{{A\oplus V}}(x_4+u_4),[x_1+u_1,x_2+u_2,x_5+u_5]_{A\oplus V}]_{A\oplus V}\\
    -&[[x_1+u_1,x_2+u_2,x_3+u_3]_{A\oplus V}^C,(\alpha+\phi)_{{A\oplus V}}(x_4+u_4),(\alpha+\phi)_{{A\oplus V}}(x_5+u_5)]_{A\oplus V}\\
    +&[[x_1+u_1,x_2+u_2,x_4+u_4]_{A\oplus V}^C,(\alpha+\phi)_{{A\oplus V}}(x_3+u_3),(\alpha+\phi)_{{A\oplus V}}(x_5+u_5)]_{A\oplus V} \\
=&\{\alpha(x_1),\alpha(x_2),\{x_3,x_4,x_5\}\}+l(\alpha(x_1),\alpha(x_2))l(x_3,x_4)\phi(u_5)
-l(\alpha(x_1),\alpha(x_2))r(x_3,x_5)\phi(u_4) \\
+&l(\alpha(x_1),\alpha(x_2))r(x_4,x_5)\phi(u_3)
-r(\alpha(x_1),\{x_3,x_4,x_5\})\phi(u_2)+r(\alpha(x_2),\{x_3,x_4,x_5\})\phi(u_1)\\
-&\{\alpha(x_3),\alpha(x_4),\{x_1,x_2,x_5\}\}-l(\alpha(x_3),\alpha(x_4))l(x_1,x_3)\phi(u_5)
+l(\alpha(x_3),\alpha(x_4))r(x_1,x_5)\phi(u_2)  \\
-& l(\alpha(x_3),\alpha(x_4))r(x_2,x_5)\phi(u_1)
+r(\alpha(x_3),\{x_1,x_2,x_5\})\phi(u_4)-r(\alpha(x_4),\{x_1,x_2,x_5\})\phi(u_3)
\\-&\{[x_1,x_2,x_3]^C,\alpha(x_4),\alpha(x_5)\}+l([x_1,x_2,x_3]^C,\alpha(x_4))\phi(u_5)-l([x_1,x_2,x_3]^C,\alpha(x_5))\phi(u_4)
\\+&\circlearrowleft_{1,2,3}\Big(r(\alpha(x_4),\alpha(x_5))l(x_1,x_2)\phi(u_3)
-r(\alpha(x_4),\alpha(x_5))r(x_1,x_3)\phi(u_2)+r(\alpha(x_4),\alpha(x_5))r(x_2,x_3)\phi(u_1)\Big)
\\+&\{[x_1,x_2,x_4]^C,\alpha(x_3),\alpha(x_5)\}-l([x_1,x_2,x_4]^C,\alpha(x_3))\phi(u_5)+l([x_1,x_2,x_4]^C,\alpha(x_5))\phi(u_3)
\\-&\circlearrowleft_{1,2,4}\Big(r(\alpha(x_3),\alpha(x_5))l(x_1,x_2)\phi(u_4)
+r(\alpha(x_3),\alpha(x_5))r(x_1,x_4)\phi(u_2)-r(\alpha(x_3),\alpha(x_5))r(x_2,x_4)\phi(u_1)\Big)=0
\end{align*}}
if and only if the identities \eqref{rep1}-\eqref{rep3}  hold.
Moreover,
{\small\begin{align*}
    &[[x_1+u_1,x_2+u_2,x_3+u_3]_{A\oplus V}^C,(\alpha+\phi)_{{A\oplus V}}(x_4+u_4),(\alpha+\phi)_{{A\oplus V}}(x_5+u_5)]_{A\oplus V}\\
    -&[(\alpha+\phi)_{{A\oplus V}}(x_1+u_1),(\alpha+\phi)_{{A\oplus V}}(x_2+u_2),[x_3+u_3,x_4+u_4,x_5+u_5]_{A\oplus V}]_{A\oplus V}\\
    -&
    [(\alpha+\phi)_{{A\oplus V}}(x_2+u_2),(\alpha+\phi)_{{A\oplus V}}(x_3+u_3),[x_1+u_1,x_4+u_4,x_5+u_5]_{A\oplus V}]_{A\oplus V}\\
    -&[(\alpha+\phi)_{{A\oplus V}}(x_3+u_3),(\alpha+\phi)_{{A\oplus V}}(x_1+u_1),[x_2+u_2,x_4+u_4,x_5+u_5]_{A\oplus V}]_{A\oplus V}\\
    =&[[x_1+u_1,x_2+u_2,x_3+u_3]_{A\oplus V}^C,(\alpha+\phi)_{{A\oplus V}}(x_4+u_4),(\alpha+\phi)_{{A\oplus V}}(x_5+u_5)]_{A\oplus V}\\
    -&\circlearrowleft_{1,2,3}[(\alpha+\phi)_{{A\oplus V}}(x_1+u_1),(\alpha+\phi)_{{A\oplus V}}(x_2+u_2),[x_3+u_3,x_4+u_4,x_5+u_5]_{A\oplus V}]_{A\oplus V}
\\=&\{[x_1,x_2,x_3]^C,\alpha(x_4),\alpha(x_5)\}-l([x_1,x_2,x_3]^C,\alpha(x_4))\phi(u_5)+l([x_1,x_2,x_3]^C,\alpha(x_5))\phi(u_4)
\\-&\circlearrowleft_{1,2,3}\Big(r(\alpha(x_4),\alpha(x_5))l(x_1,x_2)\phi(u_3)
-r(\alpha(x_4),\alpha(x_5))r(x_1,x_3)\phi(u_2)+r(\alpha(x_4),\alpha(x_5))r(x_2,x_3)\phi(u_1)\Big)\\
-&\circlearrowleft_{1,2,3}\Big(\{\alpha(x_1),\alpha(x_2),\{x_3,x_4,x_5\}\}
+l(\alpha(x_1),\alpha(x_2))l(x_3,x_4)\phi(u_5)\\
-&l(\alpha(x_1),\alpha(x_2))r(x_3,x_5)\phi(u_4)
+l(\alpha(x_1),\alpha(x_2))r(x_4,x_5)\phi(u_3)\\&-r(\alpha(x_1),\{x_3,x_4,x_5\})\phi(u_2)+r(\alpha(x_2),\{x_3,x_4,x_5\})\phi(u_1)\Big).=0
\end{align*}}
if and only if the identities \eqref{rep2}-\eqref{rep4}  hold.
\end{proof}

  \begin{pro}
Let $(V,l,r,\phi)$ be a representation of a $3$-Hom-pre-Lie algebra $(A,\{\cdot,\cdot,\cdot\},\a)$. Then $\mu$ is a representation of the
sub-adjacent $3$-Hom-Lie algebra $(A^c,[\cdot,\cdot,\cdot]^C,\alpha)$ on the vector space $V$ with respect to $\phi$,
where $\mu$ is defined in definition \ref{def-rep-3-hompre}.
  \end{pro}
\begin{proof}
Straightforward.
\end{proof}

Now, we study the dual representation of a $3$-Hom-pre-Lie algebra. We give a definition
without any additional condition, similarly as for Hom-Lie algebras and $3$-Hom-Lie algebras.
Let $(V,l,r,\phi)$ be a representation of a $3$-Hom-pre-Lie algebra  $(A,\{\cdot,\cdot,\cdot\},\a)$.
 Define $l^*: A\wedge A \longrightarrow gl(V^*)$  and $r^*: A \otimes A \longrightarrow gl(V^*)$ as usual by
 \begin{align*}
\langle l^*(x,y)(\xi),u\rangle=&-\langle\xi,l(x,y)(u)\rangle,\\
\langle r^*(x,y)(\xi),u\rangle=&-\langle\xi,r(x,y)(u)\rangle,
\end{align*}
for all $x,y \in A,u\in V,\xi\in V^*$.
However, in general $(l^*-r^*\tau+r^*,-r^*)$ is not a representation of $A$ anymore.  We  will try to construct the dual representation of a representation of a 3-Hom-pre-Lie algebra without any condition.

Define $l^\star:A \wedge A \longrightarrow gl(V^*)$ and $r^\star :A \otimes A \longrightarrow gl(V^*)$ by
\begin{align}\label{l,r star 1}
 l^\star(x,y)(\xi):=&l^*(\a(x),\a(y))\big{(}(\phi^{-2})^*(\xi)\big{)}, \\
  r^\star(x,y)(\xi):=&r^*(\a(x),\a(y))\big{(}(\phi^{-2})^*(\xi)\big{)}.
\end{align}
More precisely,  we have
\begin{align}\label{l,r star 2}
\langle l^\star(x,y)(\xi),u\rangle=&-\langle\xi,l(\a^{-1}(x),\a^{-1}(y))(\phi^{-2}(u))\rangle, \\
\langle r^\star(x,y)(\xi),u\rangle=&-\langle\xi,r(\a^{-1}(x),\a^{-1}(y))(\phi^{-2}(u))\rangle,
\end{align}
for any $x,y \in A, u\in V, \xi \in V^*$.

\begin{pro}
Let $(V,l,r,\phi)$ be a representation  of a $3$-Hom-pre-Lie algebra $(A,\{\cdot,\cdot,\cdot\},\a)$.  Then $(\mu^\star,-r^\star)$ is a representation of the  $3$-Hom-pre-Lie algebra $(A,\{\cdot,\cdot,\cdot\},\a)$  on the vector space $V^*$ with respect to $(\phi^{-1})^*$, which is called the dual representation of the representation $(V,l,r,\phi)$.
\end{pro}

\begin{proof}
Straightforward.
\end{proof}

 If $(V,l,r,\phi)=(A,L,R,\a)$ is the adjoint representation of a $3$-Hom-pre-Lie algebra $(A,\{\cdot,\cdot,\cdot\},\a)$, then we obtain
$(l^\star-r^\star\tau+r^\star,-r^\star)=(ad^\star,-R^\star)$ and $(\phi^{-1})^*=(\a^{-1})^*$.

\begin{defi}
 Let $(V,l,r,\phi)$ be a representation of a $3$-Hom-pre-Lie algebra $(A,\{\c,\c,\c\},\alpha)$. A linear operator $T: V \to A$ is called and $\mathcal{O}$-operator associated to $(V,l,r,\phi)$ if $T$ satisfies
 \begin{eqnarray}\label{O-op 3-pre-Lie}
   &T\phi=\alpha T.&\nonumber\\
   & \{Tu,Tv,Tw\}=T\left(l(Tu,Tv)w-r(Tu,Tw)v+r(Tv,Tw)u\right),\quad \forall u,v,w\in V.&
 \end{eqnarray}
\end{defi}
If $V=A$, then $T$ is called a Rota-Baxter operator on $A$ of weight zero. That is
\begin{align*}
\{\mathcal R(x),\mathcal R(y),\mathcal R(z)\}= \mathcal R\big(\{\mathcal R(x),\mathcal R(y),z\}+\{\mathcal R(x),y,\mathcal R(z)\}+\{x,\mathcal R(y),\mathcal R(z)\}\big),
\end{align*}
for all $x,y,z \in A$.
%

%%%%%%%%%%%%%%%%%%%%%%%%%%%%%%%%%%%
\section{$3$-Hom-L-dendriform algebras}
%%%%%%%%%%%%%%%%%%%%%%%%%%%%%%%%%%%
In this section, we introduce the notion of a $3$-Hom-L-dendriform algebra which is exactly the ternary version of a Hom-L-dendriform algebra.

\begin{defi}
Let $A$ be a vector space with three linear maps $\nwarrow, \nearrow : \otimes^3 A \to A$ and $\alpha:A\rightarrow A$. The 4-uple $(A,\nwarrow,\nearrow,\alpha)$ is called a $3$-Hom-L-dendriform algebra if the following identities hold:
\begin{align}\label{3-Hom-L-dendriform0}
&  \nw(x_1,x_2,x_3)+ \nw(x_2,x_1,x_3)=0  ,\\
 \label{3-Hom-L-dendriform1}&   \nw(\alpha(x_1),\alpha(x_2),\nw(x_3,x_4,x_5))-  \nw(\alpha(x_3),\alpha(x_4),\nw(x_1,x_2,x_5)) \nonumber\\
&\hspace{2 cm} =\nw([x_1,x_2,x_3]^C,\alpha(x_4),\alpha(x_5))-\nw([x_1,x_2,x_4]^C,\alpha(x_3),\alpha(x_5))      ,\\
 &\label{3-Hom-L-dendriform2}  \nw(\alpha(x_1),\alpha(x_2),\ne(x_5,x_3,x_4))-\ne(\alpha(x_5),\alpha(x_3),\{x_1,x_2,x_4\}^h) \nonumber\\
& \hspace{2 cm} =\ne(\alpha(x_5),[x_1,x_2,x_3]^C,\alpha(x_4))+\ne(\{x_1,x_2,x_5\}^v,\alpha(x_3),\alpha(x_4))      , \\
 & \label{3-Hom-L-dendriform3}\ne(\alpha(x_5),\alpha(x_1),\{x_2,x_3,x_4\}^h)-\nw(\alpha(x_2),\alpha(x_3),\ne(x_5,x_1,x_4))
 \nonumber\\
&\hspace{2 cm}=\ne(\{x_1,x_2,x_5\}^v,\alpha(x_3),\alpha(x_4))-\ne(\{x_1,x_3,x_5\}^v,\alpha(x_2),\alpha(x_4))
      ,\\
 &  \label{3-Hom-L-dendriform4}\nw([x_1,x_2,x_3]^C,\alpha(x_4),\alpha(x_5))=\circlearrowleft_{1,2,3}\nw(\alpha(x_1),\alpha(x_2),\nw(x_3,x_4,x_5)),
   \\
 &   \label{3-Hom-L-dendriform5}     \ne(\alpha(x_5),[x_1,x_2,x_3]^C,\alpha(x_4))=\circlearrowleft_{1,2,3}\nw(\alpha(x_1),\alpha(x_2),\ne(x_5,x_3,x_4)) ,\\
 &   \label{3-Hom-L-dendriform6}   \nw(\alpha(x_1),\alpha(x_2),\ne(x_5,x_3,x_4))+\ne(\alpha(x_5),\alpha(x_1),\{x_2,x_3,x_4\}^h)\nonumber\\
&\hspace{2 cm} =\ne(\alpha(x_5),\alpha(x_2),\{x_1,x_3,x_4\}^h)+\ne(\{x_1,x_2,x_5\}^v,\alpha(x_3),\alpha(x_4))     ,
\end{align}
 for all $x_i \in A$, $1\leq i \leq 5$, where
\begin{align}
& \{x,y,z\}^h=\nw(x,y,z)+\ne(x,y,z)-\ne(y,x,z)  ,   \label{accolade horizintal}\\
& \{x,y,z\}^v=\nw(x,y,z)+\ne(z,x,y)-\ne(z,y,x)  ,   \label{accolade vertical}\\
& [x,y,z]^C= \circlearrowleft_{x,y,z} \{x,y,z\}^h=\circlearrowleft_{x,y,z} \{x,y,z\}^v, \label{crochet}
\end{align}
for an $x,y,z \in A$.
\end{defi}

A morphism of $3$-Hom-L-dendriform algebra $f:(A_1,\nwarrow_1,\nearrow_1,\alpha_1)\rightarrow(A_2,\nwarrow_2,\nearrow_2,\alpha_2)$ is a linear map such that:
\begin{eqnarray*}
f\nwarrow_1(x_1,x_2,x_3)&=&\nwarrow_2(f(x_1),f(x_2),f(x_3)),\\
f\nearrow_1(x_1,x_2,x_3)&=&\nearrow_2(f(x_1),f(x_2),f(x_3)),\\
f\circ\alpha_1&=&\alpha_2\circ f,
\end{eqnarray*}
for all $x_1,x_2,x_3\in A_1$.
\begin{pro}
Let $(A,\nw,\ne)$ be a $3$-L-dendriform algebra. Define two linear maps $\nwarrow_\alpha , \nearrow_\alpha : \otimes^3 A \to A$ by
$$
\nwarrow_\alpha(x_1,x_2,x_3)=\nwarrow(\alpha(x_1),\alpha(x_2),\alpha(x_3));~~~~\nearrow_\alpha(x_1,x_2,x_3)=\nearrow(\alpha(x_1),\alpha(x_2),\alpha(x_3)),\forall x_1,x_2,x_3\in A.
$$
Then $(A,\nw_\alpha,\ne_\alpha,\alpha)$ is $3$-Hom-L-dendriform algebra.

\end{pro}
\begin{proof}
For any $x_1,x_2,x_3\in A$, we have:
$$\nw_\alpha(x_1,x_2,x_3)+ \nw_\alpha(x_2,x_1,x_3)=\nw(\alpha(x_1),\alpha(x_2),\alpha(x_3))+ \nw(\alpha(x_2),\alpha(x_1),\alpha(x_3))=0.$$ Moreover, we have
 for  $x_1,x_2,x_3,x_4,x_5\in A$ \begin{align*}
& \nw_\alpha(\alpha(x_1),\alpha(x_2),\nw_\alpha(x_3,x_4,x_5))-  \nw_\alpha(\alpha(x_3),\alpha(x_4),\nw_\alpha(x_1,x_2,x_5)) \\
&=\nw(\alpha^2(x_1),\alpha^2(x_2),\nw(\alpha^2(x_3),\alpha^2(x_4),\alpha^2(x_5)))-  \nw(\alpha^2(x_3),\alpha^2(x_4),\nw(\alpha^2(x_1),\alpha^2(x_2),\alpha^2(x_5)))\\
&=\nw_\alpha([x_1,x_2,x_3]^C,\alpha(x_4),\alpha(x_5))-\nw_\alpha([x_1,x_2,x_4]^C,\alpha(x_3),\alpha(x_5)).
\end{align*}
Which imply  \eqref{3-Hom-L-dendriform1}.
 Similarly, we obtain the identities \eqref{3-Hom-L-dendriform2}-\eqref{3-Hom-L-dendriform6}.

\end{proof}

\begin{pro}\label{3LDendTo3PreLie}
Let $(A,\nw,\ne,\alpha)$ be a $3$-Hom-L-dendriform algebra. The bracket defined in \eqref{accolade horizintal} $($respectively \eqref{accolade vertical}$)$ defines a $3$-Hom-pre-Lie algebra structure on $A$ which is called the associated  horizontal  $($ respectively associated vertical $)$ $3$-Hom-pre-Lie algebra of  $(A,\nw,\ne,\alpha)$  and $(A,\nw,\ne,\alpha)$ is also called  a compatible $3$-Hom-L-dendriform algebra structure  on the  $3$-Hom-pre-Lie algebra $(A,\{\c,\c,\c\}^h,\alpha)$\\ $($ respectively $(A,\{\c,\c,\c\}^v,\alpha)$ $)$.
\end{pro}

\begin{proof}
We will just prove that $(A,\{\c,\c,\c\}^h,\alpha)$ is a $3$-Hom-pre-Lie algebra.\\
Note, first that $\{x,y,z\}^h=-\{y,x,z\}^h$, for any $x,y,z \in A$.
Let $x_i \in A,\ 1 \leq i \leq 5$, then
\begin{align*}
&\{\alpha(x_1),\alpha(x_2),\{x_3,x_4,x_5\}^h\}^h-\{\alpha(x_3),\alpha(x_4),\{x_1,x_2,x_5\}^h\}^h \\
&-\{[x_1,x_2,x_3]^C,\alpha(x_4),\alpha(x_5)\}^h+\{[x_1,x_2,x_4]^C,\alpha(x_3),\alpha(x_5)\}^h\\
&= r_1+ r_2+r_3+r_4+r_5,
\end{align*}
where
\begin{align*}
r_1&= \nw(\alpha(x_1),\alpha(x_2),\nw(x_3,x_4,x_5))-  \nw(\alpha(x_3),\alpha(x_4),\nw(x_1,x_2,x_5))  -\nw([x_1,x_2,x_3]^C,\alpha(x_4),\alpha(x_5))\\
&+\nw([x_1,x_2,x_4]^C,\alpha(x_3),\alpha(x_5)), \\
r_2&= \ne(\alpha(x_3),[x_1,x_2,x_4]^C,\alpha(x_5))+\ne(\alpha(x_3),\alpha(x_4),\{x_1,x_2,x_5\}^h) - \nw(\alpha(x_1),\alpha(x_2),\ne(x_3,x_4,x_5))  \\
&+\ne(\{x_1,x_2,x_3\}^v,\alpha(x_4),\alpha(x_5)) ,\\
r_3&= \ne(\alpha(x_4),[x_1,x_2,x_3]^C,\alpha(x_5))+\ne(\alpha(x_4),\alpha(x_3),\{x_1,x_2,x_5\}^h)- \nw(\alpha(x_1),\alpha(x_2),\ne(x_4,x_3,x_5))\\
&  +\ne(\{x_1,x_2,x_4\}^v,\alpha(x_3),\alpha(x_5)),\\
r_4&= \ne(\alpha(x_1),\alpha(x_2),\{x_3,x_4,x_5\}^h)-\nw(\alpha(x_3),\alpha(x_4),\ne(x_1,x_2,x_5))
 \nonumber-\ne(\{x_2,x_3,x_1\}^v,\alpha(x_4),\alpha(x_5))
 \\
&+\ne(\{x_2,x_4,x_1\}^v,\alpha(x_3),\alpha(x_5)) ,\\
 r_5&= \ne(\alpha(x_2),\alpha(x_1),\{x_3,x_4,x_5\}^h)-\nw(\alpha(x_3),\alpha(x_4),\ne(x_2,x_1,x_5))
 \nonumber-\ne(\{x_1,x_3,x_2\}^v,\alpha(x_4),\alpha(x_5))
\\
& +\ne(\{x_1,x_4,x_2\}^v,\alpha(x_3),\alpha(x_5)).
\end{align*}
From identities \eqref{3-Hom-L-dendriform1}-\eqref{3-Hom-L-dendriform3}, we obtain immediately $r_i=0,\ \forall 1 \leq i \leq 5$.
This implies that Eq.\eqref{3-hom-pre-Lie 1} holds.

On the other hand, we have
{\small\begin{align*}
&\{ [x_1,x_2,x_3]^C,\alpha(x_4), \alpha(x_5)\}^h-\{\alpha(x_1),\alpha(x_2),\{ x_3,x_4, x_5\}^h\}^h-\{\alpha(x_2),\alpha(x_3),\{ x_1,x_4,x_5\}^h\}^h -\{\alpha(x_3),\alpha(x_1),\{ x_2,x_4, x_5\}^h\}^h \\
=& s_1+s_2+s_3+s_4+s_5,
\end{align*}}
where
\begin{align*}
&s_1=\nw([x_1,x_2,x_3]^C,\alpha(x_4),\alpha(x_5))-
\circlearrowleft_{1,2,3}\nw(\alpha(x_1),\alpha(x_2),\nw(x_3,x_4,x_5)),\\
&s_2= \ne(\alpha(x_4),[x_1,x_2,x_3]^C,\alpha(x_5))-
\circlearrowleft_{1,2,3}\nw(\alpha(x_1),\alpha(x_2),\ne(x_4,x_3,x_5)) ,\\
&s_3=\nw(\alpha(x_1),\alpha(x_2),\ne(x_3,x_4,x_5))+\ne(\alpha(x_3),\alpha(x_1),\{x_2,x_3,x_5\}^h)\\
&\hspace{2 cm} -\ne(\alpha(x_3),\alpha(x_2),\{x_1,x_4,x_5\}^h)-\ne(\{x_1,x_2,x_3\}^v,\alpha(x_4),\alpha(x_5)),\\
&s_4= \nw(\alpha(x_2),\alpha(x_3),\ne(x_1,x_4,x_5))+\ne(\alpha(x_1),\alpha(x_2),\{x_3,x_4,x_5\}^h)\\
&\hspace{2 cm} -\ne(\alpha(x_1),\alpha(x_3),\{x_2,x_4,x_5\}^h)-\ne(\{x_2,x_3,x_1\}^v,\alpha(x_4),\alpha(x_5)),\\
&s_5=  \nw(\alpha(x_3),\alpha(x_1),\ne(x_2,x_4,x_5))+\ne(\alpha(x_2),\alpha(x_3),\{x_1,x_4,x_5\}^h)\\
&\hspace{2 cm} -\ne(\alpha(x_2),\alpha(x_1),\{x_3,x_4,x_5\}^h)-\ne(\{x_3,x_1,x_2\}^v,\alpha(x_4),\alpha(x_5)).
\end{align*}
From identities \eqref{3-Hom-L-dendriform4}-\eqref{3-Hom-L-dendriform6}, we obtain immediately $s_i=0,\ \forall 1 \leq i \leq 5$.
This implies that Eq.\eqref{3-hom-pre-Lie 2} holds.
\end{proof}
It is straightforward to verify that:
\begin{cor}
Let $(A,\nw,\ne,\alpha)$ be  a $3$-Hom-L-dendriform algebra. Then the bracket given by  \eqref{crochet} together with $\alpha$ define a $3$-Hom-Lie algebra structure on $A$ which is called the associated   $3$-Hom-Lie algebra of  $(A,\nw,\ne,\alpha)$.
\end{cor}
The following Proposition is obvious.
\begin{pro}
Let $(A,\nw,\ne,\alpha)$ be a $3$-Hom-L-dendriform algebra. Define $L_{\nw},R_{\ne}: \otimes^2 A \to gl(A)$ by
$$L_{\nw}(x,y)z=\nw(x,y,z),\quad  R_{\ne}(x,y)z=\ne(z,x,y),\ \forall x,y,z \in A.$$
Then $(A,L_{\nw},R_{\ne},\alpha)$ is a representation of  its  associated   $3$-Hom-pre-Lie algebra $(A,\{\c,\c,\c\},\alpha)$.
\end{pro}

\begin{thm}\label{3HomLden by O-op}
 Let $(A,\{\c,\c,\c\},\alpha)$ be a $3$-Hom-pre-Lie algebra and $(V,l,r,\phi)$ be a representation. Suppose that  $T: V \to A$ is an $\mathcal{O}$-operator associated to $(V,l,r,\phi)$. Then $(V,\nw,\ne,\phi)$ is a $3$-Hom-L-dendriform algebra, where
 \begin{align}
  \nw(u,v,w)=l(Tu,Tv)w  , \quad \ne(u, v,w)= r(Tv,Tw)u, \forall u,v,w \in V.
 \end{align}
 Therefore, there exists an associated $3$-Hom-pre-Lie algebra structure on $V$ and $T$ is a homomorphism of $3$-Hom-pre-Lie algebras. Moreover, $T(V)=\{T(v)| v \in V \}$ is $3$-Hom-pre-Lie subalgebra of $(A,\{\c,\c,\c\},\alpha)$ and there is an induced $3$-Hom-L-dendriform algebra structure on $T(V)$ given by
  \begin{align}
  \nw(Tu,Tv,Tw)=T(  \nw(u,v,w))  , \quad \ne(Tu, Tv,Tw)= T(\ne(u, v,w)), \forall u,v,w \in V.
 \end{align}

\end{thm}

\begin{proof}
  Define $\{\c,\c,\c\}_V^h, \{\c,\c,\c\}_V^v,[\c,\c,\c]^C_V:\otimes^3V \to V$ by
\begin{align*}
& \{u,v,w\}_V^h=\nw(u,v,w)+\ne(u,v,w)-\ne(v,u,w),\\
& \{u,v,w\}_V^v=\nw(u,v,w)+\ne(w,u,v)-\ne(w,v,u),\\
&[u,v,w]^C_V=\circlearrowleft_{u,v,w}\{u,v,w\}_V^h=\circlearrowleft_{u,v,w}\{u,v,w\}_V^v.
\end{align*}
for any $u,v,w \in V$.\\
Using identity \eqref{O-op 3-pre-Lie}, we have
\begin{align*}
T\{u,v,w\}_V^h &=T(\nw(u,v,w)+\ne(u,v,w)-\ne(v,u,w))\\
&=T(l(Tu,Tv)w-r(Tu,Tw)v+r(Tv,Tw)u)=\{Tu,Tv,Tw\}^h
\end{align*}
and
$$T[u,v,w]^C_V=\circlearrowleft_{u,v,w}T\{u,v,w\}_V
=\circlearrowleft_{u,v,w}\{Tu,Tv,Tw\}^h=[Tu,Tv,Tw]^C.$$
It is obvious that
$$\nw(u,v,w)+ \nw(v,u,w)= ( l(Tu,Tv)+l(Tv,Tu))w=0.$$
Furthermore, for any $u_i \in V,\ 1 \leq i \leq 5$, we have
\begin{align*}
&   \nw(\phi(u_1),\phi(u_2),\nw(u_3,u_4,u_5))-  \nw(\phi(u_3),\phi(u_4),\nw(u_1,u_2,u_5)) \\
&-\nw([u_1,u_2,u_3]_V^C,\phi(u_4),\phi(u_5))+\nw([u_1,u_2,u_4]_V^C,\phi(u_3),\phi(u_5)) \\
&= l(T(\phi(u_1)),T(\phi(u_2)))l(T(u_3),T(u_4))u_5-  l(T(\phi(u_3)),T(\phi(u_4)))l(T(u_1),T(u_2))u_5 \\
&-l(T[u_1,u_2,u_3]_V^C,T(\phi(u_4)))\phi(u_5)+l(T[u_1,u_2,u_4]_V^C,T(\phi(u_3)))\phi(u_5) \\
&=l(\alpha(T(u_1)),\alpha(T(u_2)))l(T(u_3),T(u_4))u_5-  l(\alpha(T(u_3)),\alpha(T(u_4)))l(T(u_1),T(u_2))u_5 \\
&-l(T[u_1,u_2,u_3]_V^C,\alpha(T(u_4)))\phi(u_5)+l(T[u_1,u_2,u_4]_V^C,\alpha(T(u_3)))\phi(u_5) =0.
\end{align*}
This implies that \eqref{3-Hom-L-dendriform1} holds.
Moreover, \eqref{3-Hom-L-dendriform2} holds. Indeed,
\begin{align*}
& \ne(\phi(u_5),[u_1,u_2,u_3]_V^C,\phi(u_4))-\ne(\phi(u_5),\phi(u_3),\{u_1,u_2,u_4\}^h_V) \\
&- \nw(\phi(u_1),\phi(u_2),\ne(u_5,u_3,u_4)) +\ne(\{u_1,u_2,u_5\}^v_V,\phi(u_3),\phi(u_4)) \\
&=r(T[u_1,u_2,u_3]_V^C,T(\phi(u_4)))\phi(u_5)-r(T(\phi(u_3)),T\{u_1,u_2,u_4\}^h_V)\phi(u_5) \\
&-l(T(\phi(u_1)),T(\phi(u_2)))r(T(u_3),T(u_4))u_5 +r(T(\phi(u_3)),T(\phi(u_4)))\{u_1,u_2,u_5\}^v_V \\
&=r([Tu_1,Tu_2,Tu_3]^C,\alpha(T(u_4)))\phi(u_5)-r(\alpha(T(u_3)),\{T(u_1),T(u_2),T(u_4)\}^h)\phi(u_5) \\
&-l(\alpha(T(u_1)),\alpha(T(u_2)))r(T(u_3),T(u_4))u_5 +r(\alpha(T(u_3)),\alpha(T(u_4)))\mu(T(u_1),T(u_2))u_5=0.
\end{align*}
To prove identity \eqref{3-Hom-L-dendriform6}, we compute as follows
\begin{align*}
& \nw(\phi(u_1),\phi(u_2),\ne(u_5,u_3,u_4))+\ne(\phi(u_5),\phi(u_1),\{u_2,u_3,u_4\}^h_V) \\
&\hspace{1 cm} -\ne(\phi(u_5),\phi(u_2),\{u_1,u_3,u_4\}^h_V)-\ne(\{u_1,u_2,u_5\}^v_V,\phi(u_3),\phi(u_4))  \\
    &=l(\alpha(Tu_1),\alpha(Tu_2))r(Tu_3,Tu_4)u_5+r(\alpha(Tu_1),\{Tu_2,Tu_3,Tu_4\}^h)\phi(u_5)  \\
    & \hspace{1 cm} -r(\alpha(Tu_2),\{Tu_1,Tu_3,Tu_4\}^h)\phi(u_5)-r(\alpha(Tu_3),\alpha(Tu_4))\mu(Tu_1,Tu_2)u_5=0.
\end{align*}
The other identities can be shown similarly.
\end{proof}
\begin{cor}\label{3HomLden by invert O-op}
   Let $(A,\{\c,\c,\c\},\alpha)$ be a $3$-Hom-pre-Lie algebra and  $\mathcal R: A \to A$ is a Rota-Baxter operator of weight $0$. Then $(A,\nw,\ne,\alpha)$ is  a $3$-Hom-L-dendriform algebra, where the operations $\nw$ and $\ne$ are  given by
 \begin{align}\label{LDendViaRB}
  \nw(x,y,z)=L(\mathcal Rx,\mathcal Ry)z=\{\mathcal Rx,\mathcal Ry,z\}  &, \quad \ne(x, y,z)= R(\mathcal Ry,\mathcal Rz)x=\{x,\mathcal Ry,\mathcal Rz\},  \end{align}
for all $x,y,z \in A$.

\end{cor}

\begin{cor}\label{3HomLden by invert O-op1}
 Let $(A,\{\c,\c,\c\},\alpha)$ be a $3$-Hom-pre-Lie algebra.  Then there exists a compatible
3-Hom-L-dendriform  algebra if and only if there exists an invertible $\mathcal{O}$-operator on $A$.
\end{cor}
 \begin{proof}
 Let $T$ be an invertible $\mathcal{O}$-operator of $A$ associated to a representation $(V, l,r,\phi)$.
Then there exists a $3$-Hom-L-dendriform algebra structure on $V$ defined by
 \begin{align}
  \nw(u,v,w)=l(Tu,Tv)w  , \quad \ne(u, v,w)= r(Tv,Tw)u, \forall u,v,w \in V.
 \end{align}
In addition there exists a $3$-Hom-L-dendriform algebra structure on $T(V)=A$ given by
 \begin{align}
  \nw(Tu,Tv,Tw)=T(l(Tu,Tv)w)  , \quad \ne(Tu, Tv,Tw)= T(r(Tv,Tw)u), \forall u,v,w \in V.
 \end{align}
 If we put $x=Tu,\ y=Tv$ and $z=Tw$, we get
  \begin{align}
  \nw(x,y,z)=T(l(x,y)T^{-1}(z))  , \quad \ne(x, y,z)= T(r(y,z)T^{-1}(x)), \forall x,y,z \in A.
 \end{align}
It is a compatible $3$-Hom-L-dendriform algebra structure on $A$. Indeed,
\begin{align*}
&  \nw(x,y,z)+\ne(x,y,z)-\ne(y,x,z) \\
=& T(l(x,y)T^{-1}(z)) + T(r(y,z)T^{-1}(x))-T(r(x,z)T^{-1}(y)) \\
=& \{TT^{-1}(x),TT^{-1}(y),TT^{-1}(z)\}=\{x,y,z\}.
\end{align*}
Conversely,   let $(A,\{\c,\c,\c\},\alpha)$ be a $3$-Hom-pre-Lie algebra and $(A,\nw,\ne,\alpha)$ its compatible $3$-Hom-L-dendriform algebra.  Then the identity map $id: A \to A$ is an $\mathcal{O}$-operator of  $(A,\{\c,\c,\c\},\alpha)$ associated to $(A,L,R,\alpha)$.
 \end{proof}

\begin{defi}
Let $(A,\{\c,\c,\c\},\a)$ be a $3$-Hom-pre-Lie algebra and $B$ be a skew-symmetric bilinear form on $A$.  We say that $B$ is closed if it satisfies for any $x,y,z,w \in A$
\begin{align}\label{symplectic bilinear form}
B(\{\alpha(x),\alpha(y),\alpha(z)\},\alpha^2(w))-B(\alpha(z),[x,y,w]^C)-B(\alpha(y),\{w,x,z\})+B(\alpha(x),\{w,y,z\})=0.
\end{align}

  If in addition $B$ is nondegenerate, then $B$ is called symplectic.

A $3$-Hom-pre-Lie algebra $(A,\{\c,\c,\c\},\a)$ equipped with a symplectic form is called a symplectic $3$-Hom-pre-Lie algebra and denoted by  $(A,\{\c,\c,\c\},\a,B)$.
\end{defi}

\begin{pro}
Let $(A,\{\c,\c,\c\},\a,B)$ be a symplectic $3$-Hom-pre-Lie algebra. Then there exists a compatible $3$-Hom-L-dendriform algebra structure on $A$ given by
\begin{align}\label{3Ldend by  form}
B(\nw(\a(x),\a(y),z),\a^2(w))=B(z,[x,y,w]^C),\quad B(\ne(x,\a(y),\a(z)),\a^2(w))=-B(x,\{w,y,z\})
\end{align}for all $x,y,z,w \in A$.
\end{pro}

\begin{proof}
Define the linear map $T: A^* \to A$ by $\langle T^{-1}x,y\rangle=B(x,y)$.  Since $B$
is $\a$-symmetric
and using Eq. \eqref{symplectic bilinear form}, we obtain that $T$ is an invertible $\mathcal{O}$-operator  on $A$ associated to the coadjoint representation $(A^*, ad^\star,-R^\star,(\a^{-1})^*)$.   By Corollary  \ref{3HomLden by invert O-op1}, there exists a compatible $3$-Hom-L-dendriform algebra structure given by
  \begin{align}
  \nw(x,y,z)=T(ad^\star(x,y)T^{-1}(z))  , \quad \ne(x, y,z)= -T(R^\star(y,z)T^{-1}(x)),\forall  x,y,z \in A.
 \end{align}
Let  $x,y,z,w \in A$, then we have
\begin{align*}
    B(\nw(\a(x),\a(y),z),\a^2(w))= & B(T(ad^\star(\a(x),\a(y))T^{-1}(z)) ,\a^2(w))=\langle ad^\star(\a(x),\a(y))T^{-1}(z),\a^2(w)\rangle \\
    =& \langle T^{-1}(z),[x,x,w]^C\rangle=B(z,[x,y,w]^C)
\end{align*}
and
\begin{align*}
    B(\ne(x,\a(y),\a(z)),\a^2(w))= & -B(T(R^\star(\a(y),\a(z)))T^{-1}(x))) ,\a^2(w))=-\langle R^\star(\a(y),\a(z))T^{-1}(x),\a^2(w)\rangle \\
    = &-\langle T^{-1}(x),\{w,y,z\}\rangle=-B(x,\{w,y,z\}).
\end{align*}
The proof is finished.
\end{proof}

\begin{cor}
Let $(A,\{\c,\c,\c\},\a,B)$ be a symplectic $3$-Hom-pre-Lie algebra and $(A,[\c,\c,\c]^C,\a)$ be its associated $3$-Hom-Lie algebra.  Then there exists a $3$-Hom-pre-Lie algebraic structure $(A,\{\c,\c,\c\}',\a)$ on $A$ given by
\begin{align}
B(\{\a(x),\a(y),\a(z)\}',\a^2(w))=B(\a(z),[x,y,w]^C)-B(\a(z),\{w,x,y\})+B(\a(z),\{w,y,x\}).
\end{align}
\end{cor}

\begin{lem}\label{commuting rota-baxter op}
Let $\{\mathcal R_1,\mathcal R_2\}$ be a pair of of commuting Rota-Baxter operators of weight zero on a $3$-Hom-Lie algebra $(A, [\c,\c,\c],\a)$. Then $\mathcal R_2$ is a Rota-Baxter operator of weight zero on the associated $3$-Hom-pre-Lie algebra
defined by  $\{x,y,z\}=[\mathcal R_1(x),\mathcal R_1(y),z]$.
\end{lem}

\begin{proof}
For any $x,y,z \in A$, we have
\begin{align*}
&\{\mathcal R_2(x),\mathcal R_2(y),\mathcal R_2(z)\}=[\mathcal R_1\mathcal R_2(x),\mathcal R_1\mathcal R_2(y),\mathcal R_2(z)] =[\mathcal R_2\mathcal R_1(x),\mathcal R_2\mathcal R_1(y),\mathcal R_2(z)] \\
&=\mathcal R_2([\mathcal R_2\mathcal R_1(x),\mathcal R_2\mathcal R_1(y),z]+[\mathcal R_1(x),\mathcal R_2\mathcal R_1(y),\mathcal R_2(z)]  +
[\mathcal R_2\mathcal R_1(x),\mathcal R_1(y),\mathcal R_2(z)] ) \\
&=\mathcal R_2(\{\mathcal R_2(x),\mathcal R_2(y),z\}+\{x,\mathcal R_2(y),\mathcal R_2(z)\}+\{\mathcal R_2(x),y,\mathcal R_2(z)\}).
\end{align*}
Hence $\mathcal R_2$ is a Rota-Baxter operator of weight zero on the $3$-Hom-pre-Lie algebra
$(A,\{\c,\c,\c\},\a)$.
\end{proof}

\begin{pro}
Let $\{\mathcal R_1,\mathcal R_2\}$ be a pair of commuting Rota-Baxter operators of weight zero on a $3$-Hom-Lie algebra $(A, [\c,\c,\c],\a)$. Then $(A,\nw,\ne,\a)$  is a $3$-Hom-L-dendriform algebra, where the two brackets are  defined by:
\begin{align*}
\nw(x,y,z)=[\mathcal R_1\mathcal R_2(x),\mathcal R_1\mathcal R_2(y),z],\quad \ne(x,y,z)=[\mathcal R_1(x),\mathcal R_1\mathcal R_2(y),\mathcal R_2(z)], \forall x,y,z \in A.
\end{align*}
\end{pro}

\begin{proof}
It follows immediately from Lemma \ref{commuting rota-baxter op} and Corollary \ref{3HomLden by invert O-op}.
\end{proof}

\begin{rem}
Let $(A,[\c,\c],\a)$ be a Hom-Lie algebra. Recall that a trace function $\tau: A \to \mathbb{K}$  is a linear map such that $\tau([x,y])=0$ and $\tau\a=\tau$, $ \forall  x,y \in A$. When $\tau$ is a trace function, it is well known \cite{Arnlind&Makhlouf&Silvestrov} that $(A,[\c,\c,\c]_{\tau},\a)$ is a $3$-Hom-Lie algebra, where
$$[x,y,z]_{\tau}:=\circlearrowleft_{x,y,z \in A}\tau(x)[y,z],\  \forall x,y,z \in A.$$
Now, let $\{\mathcal R_1,\mathcal R_2\}$ be a pair of commuting Rota-Baxter operators of weight zero on the $3$-Hom-Lie algebra $(A, [\c,\c,\c]_{\tau},\a)$. Then we can construct a $3$-Hom-L-dendriform structure on $A$, given by
\begin{align*}
 \nw(x,y,z)=& \tau (\mathcal R_1\mathcal R_2(x))[\mathcal R_1\mathcal R_2(y),z]+ \tau (\mathcal R_1\mathcal R_2(y))[z,\mathcal R_1\mathcal R_2(x)]+
      +\tau(z)[\mathcal R_1\mathcal R_2(x),\mathcal R_1\mathcal R_2(y)], \\
\ne(x,y,z)=& \tau (\mathcal R_1(x))[\mathcal R_1\mathcal R_2(y),\mathcal R_2(z)]+\tau(\mathcal R_1\mathcal R_2(y))[\mathcal R_2(z),\mathcal R_1(x)]
  + \tau(\mathcal R_2(z)) [\mathcal R_1(x),\mathcal R_1\mathcal R_2(y)],
\end{align*}
for any $x,y,z \in A$.
\end{rem}
\section{Nijenhuis Operators on $3$-Hom-L-dendriform algebras}

In this section, we study the second order deformation of $3$-Hom-L-dendriform algebras  and introduce the notion of Nijenhuis operator on $3$-Hom-L-dendriform algebras, which could generate a trivial deformation. Moreover,  we give some properties and results of Nijenhuis operators.

\subsection{Second-order deformation of $3$-Hom-L-dendriform algebras}
Let $(A,\nw,\ne,\alpha)$ be a $3$-Hom-L-dendriform algebra and $\omega_\nw^i,\omega_\ne^i:A\times A \times A\longrightarrow A,\;i=1,2$ be  a multilinear maps. Consider a $\lambda$-parameterized
family of $3$-linear operations:
\begin{equation}\label{omegaparam}
\nw_\lambda(x_1,x_2,x_3)=\nw(x_1,x_2,x_3)+\lambda\omega_\nw^1(x_1,x_2,x_3)+\lambda^2\omega_\nw^2(x_1,x_2,x_3)+\dots,
\end{equation}\begin{equation}\label{omegaparam1}
\ne_\lambda(x_1,x_2,x_3)=\ne(x_1,x_2,x_3)+\lambda\omega_\ne^1(x_1,x_2,x_3)+\lambda^2\omega_\ne^2(x_1,x_2,x_3)+\dots,
\end{equation}
where $\lambda\in \mathbb{K}$. If all $(\nw_\lambda,\ne_\lambda,\alpha)$ are $3$-Hom-L-dendriform algebra structures, we say
that $\omega_\nw^i$, $\omega_\nw^i$ generates  a  $\lambda$-parameter deformation of the $3$-Hom-L-dendriform algebra $(A,\nw,\ne,\alpha)$.
\begin{rem}
\begin{enumerate}
\item If $\lambda^2=0$, the deformation is called infinitesimal.
\item If $\lambda^n=0$,  the deformation is said to be of order $n-1$.
\end{enumerate}
\end{rem}
\begin{defi}
  A deformation is said to be trivial if there exists an linear map $N:A\longrightarrow A$ such that for all $\lambda,\;T_\lambda=id+\lambda N$ satisfies
\begin{align}
T_\lambda\circ\alpha&=\alpha\circ T_\lambda,\label{alphabetatrivialdeform}\\
T_\lambda\nw_\lambda(x_1,x_2,x_3)&=\nw(T_\lambda x_1,T_\lambda x_2,T_\lambda x_3),\label{trivialdeformation1}\\
T_\lambda\ne_\lambda(x_1,x_2,x_3)&=\ne(T_\lambda x_1,T_\lambda x_2,T_\lambda x_3),\forall x_1,x_2,x_3\in A. \label{trivialdeformation2}
\end{align}
\end{defi}
Eq. \eqref{alphabetatrivialdeform} is equals to $N\circ\alpha=\alpha\circ N.$
The left hand side of Eq. \eqref{trivialdeformation1} is equivalent to
\begin{eqnarray*}
&&\nw(x_1,x_2,x_3)+\lambda(\omega_\nw^1(x_1,x_2,x_3)+N\nw(x_1,x_2,x_3))\\&&+
\lambda^2(\omega_\nw^2(x_1,x_2,x_3)+N\omega_\nw^1(x_1,x_2,x_3))+\lambda^3N\omega_\nw^2(x_1,x_2,x_3)+\dots.
\end{eqnarray*}
The right hand side of Eq. \eqref{trivialdeformation1} can be written as
\begin{eqnarray*}
&&\nw(x_1,x_2,x_3)+
\lambda(\nw(Nx_1,x_2,x_3+\nw(x_1,Nx_2,x_3)
+\nw(x_1,x_2,Nx_3))\\&&
+\lambda^2(\nw(Nx_1,Nx_2,x_3)+\nw(Nx_1,x_2,Nx_3)+\nw(x_1,Nx_2,Nx_3))+
\lambda^3\nw(Nx_1,Nx_2,Nx_3).
\end{eqnarray*}
Therefore, by Eq. \eqref{trivialdeformation1}, we have
{\small\begin{align*}
&\omega_\nw^1(x_1,x_2,x_3)+N\nw(x_1,x_2,x_3)=\nw(Nx_1,x_2,x_3)+
\nw(x_1,Nx_2,x_3)+\nw(x_1,x_2,Nx_3),\\
&\omega_\nw^2(x_1,x_2,x_3)+N\omega_\nw^1(x_1,x_2,x_3)
=\nw(Nx_1,Nx_2,x_3)+\nw(Nx_1,x_2,Nx_3)+\nw(x_1,Nx_2,Nx_3),\\
&N\omega_\nw^2(x_1,x_2,x_3)=\nw(Nx_1,Nx_2,Nx_3).
\end{align*}}
Similarly, using \eqref{trivialdeformation2} we obtain
{\small\begin{align*}
&\omega_\ne^1(x_1,x_2,x_3)+N\ne(x_1,x_2,x_3)=\ne(Nx_1,x_2,x_3)+
\ne(x_1,Nx_2,x_3)+\ne(x_1,x_2,Nx_3),\\
&\omega_\ne^2(x_1,x_2,x_3)+N\omega_\ne^1(x_1,x_2,x_3)
=\ne(Nx_1,Nx_2,x_3)+\ne(Nx_1,x_2,Nx_3)+\ne(x_1,Nx_2,Nx_3),\\
&N\omega_\ne^2(x_1,x_2,x_3)=\ne(Nx_1,Nx_2,Nx_3).
\end{align*}}

Let $(A,\nw,\ne,\alpha)$ be a $3$-Hom-L-dendriform algebra, and $N:A\longrightarrow A$ a linear map. Define a ternary operations $\nw^1_N,\ne^1_N:\otimes^3 A \longrightarrow A $ by
{\small\begin{equation}\label{Nw1N}
\nw^1_N(x_1,x_2,x_3)=\nw(Nx_1,x_2,x_3)+\nw(x_1,Nx_2,x_3)+
\nw(x_1,x_2,Nx_3)-
N\nw(x_1,x_2,x_3),
\end{equation}\begin{equation}\label{Ne1N}
\ne^1_N(x_1,x_2,x_3)=\ne(Nx_1,x_2,x_3)+\ne(x_1,Nx_2,x_3)+
\ne(x_1,x_2,Nx_3)-
N\ne(x_1,x_2,x_3).
\end{equation}}
Then we define ternary operations  $\nw^2_N,\ne^2_N:\otimes^3 A \longrightarrow A $, via induction by
{\small\begin{equation}
\nw^2_N(x_1,x_2,x_3)=\nw(Nx_1,Nx_2,x_3)+\nw(Nx_1,x_2,Nx_3)+
\nw(x_1,Nx_2,Nx_3)-N\nw^1_N(x_1,x_2,x_3).
\end{equation}\begin{equation}
\ne^2_N(x_1,x_2,x_3)=\ne(Nx_1,Nx_2,x_3)+\ne(Nx_1,x_2,Nx_3)+
\ne(x_1,Nx_2,Nx_3)-N\ne^1_N(x_1,x_2,x_3).
\end{equation}}
\begin{defi}
Let $(A,\nw,\ne,\alpha)$ be a $3$-Hom-L-dendriform algebra. A linear map $N:A\rightarrow A$ is called a Nijenhuis operator
if
\begin{eqnarray}
\alpha\circ N&=& N\circ\alpha,\\
\label{equNijenhuis1}
\nw(N(x_1),N(x_2),N(x_3))&=&\displaystyle\sum_{ I\subseteq[3]^*}(-1)^{|I|-1}N^{|I|}\nw(\widetilde{N}(x_1),\widetilde{N}(x_2),\widetilde{N}(x_3)),
\\
\label{equNijenhuis2}
\ne(N(x_1),N(x_2),N(x_3))&=&\displaystyle\sum_{\emptyset\neq I\subseteq[3]^*}(-1)^{|I|-1}N^{|I|}\ne(\widetilde{N}(x_1),\widetilde{N}(x_2),\widetilde{N}(x_3)),
\end{eqnarray}
$\forall x_1,x_2,x_3\in A$, where
$\widetilde{N}(x_i)=\left\{
                       \begin{array}{ll}
                         x_i &\; i\;\in I \hbox{ ;} \\
                         N(x_i) & \; i\;\not\in I\hbox{.}
                       \end{array}
                     \right.$
\end{defi}
The Eqs. \eqref{equNijenhuis1} and \eqref{equNijenhuis2} can be written
\begin{eqnarray*}
\nw(Nx,Ny,Nz)&=&N(\nw(Nx,Ny,z)+\nw(Nx,y,Nz)+\nw(x,Ny,Nz)\nonumber\\
\label{equNijenhuis11}&&-N(\nw(Nx,y,z)+\nw(x,Ny,z)+\nw(x,y,Nz))+N^2\nw(x,y,z)),
\\
\ne(N(x),N(y),N(z))&=&N(\ne(Nx,Ny,z)+\ne(Nx,y,Nz)+\ne(x,Ny,Nz)\nonumber\\
\label{equNijenhuis12}&&-N(\ne(Nx,y,z)+\ne(x,Ny,z)+\ne(x,y,Nz))+N^2\ne(x,y,z)).
\end{eqnarray*}

\subsection{Some properties of Nijenhuis operators}
\begin{pro}
  Let $N$ be a Nijenhuis operator on a $3$-Hom-L-dendriform algebra $(A,\nw,\ne,\alpha)$. Then $(A,\nw^1_N,\ne^1_N,\alpha)$ is a $3$-Hom-L-dendriform algebra and $N$ is a morphism of $3$-Hom-L-dendriform algebra from  $(A,\nw^1_N,\ne^1_N,\alpha)$ to  $(A,\nw,\ne,\alpha)$, where the two operations $\nw^1_N,\ne^1_N$ are defined above by equations \eqref{Nw1N} and \eqref{Ne1N}.
\end{pro}
\begin{proof}Straightforward.
\end{proof}
\begin{pro}
Let $N$ be a Nijenhuis operator on a $3$-Hom-L-dendriform algebras $(A,\nw,\ne,\alpha)$. Then $N$ is a  Nijenhuis operator on their sub-adjacent  $3$-Hom-pre-Lie algebras  $(A,\{\cdot,\cdot,\cdot\}^h,\alpha)$, $(A,\{\cdot,\cdot,\cdot\}^v,\alpha)$ and its  sub-adjacent  $3$-Hom-Lie algebra  $(A,[\cdot,\cdot,\cdot]^C,\alpha)$.

\end{pro}
\begin{proof}  Let $x,y,z\in A$, then we have
\begin{align*}
   \{Nx,Ny,Nz\}^h=&\nw(Nx,Ny,Nz)+\ne(Nx,Ny,Nz)-\ne(Ny,Nx,Nz)  \\
   =&N(\nw(Nx,Ny,z)+\nw(Nx,y,Nz)+\nw(x,Ny,Nz)\\
&-N(\nw(Nx,y,z)+\nw(x,Ny,z)+\nw(x,y,Nz))+N^2\nw(x,y,z))\\
&+N(\ne(Nx,Ny,z)+\ne(Nx,y,Nz)+\ne(x,Ny,Nz)\nonumber\\
&-N(\ne(Nx,y,z)+\ne(x,Ny,z)+\ne(x,y,Nz))+N^2\ne(x,y,z))\\
&-N(\ne(Ny,Nx,z)-\ne(Ny,x,Nz)-\ne(y,Nx,Nz)\\\
&+N(\ne(Ny,x,z)+\ne(y,Nx,z)+\ne(y,x,Nz))-N^2\ne(y,x,z)))\\
=&N(\{Nx,Ny,z\}^h+\{Nx,y,Nz\}^h+\{x,Ny,Nz\}^h\nonumber\\
&-N(\{Nx,y,z\}^h+\{x,Ny,z\}^h+\{x,y,Nz\}^h)+N^2\{x,y,z\}^h).
\end{align*}
Therefor  $N$ is a Nijenhuis operator on  sub-adjacent  $3$-Hom-pre-Lie algebras  $(A,\{\cdot,\cdot,\cdot\}^h,\alpha)$. Similarly, we can found that $N$ is a Nijenhuis operator on  $3$-Hom-pre-Lie algebras  $(A,\{\cdot,\cdot,\cdot\}^v,\alpha)$ and on the  sub-adjacent  $3$-Hom-Lie algebra  $(A,[\cdot,\cdot,\cdot]^C,\alpha)$.
\end{proof}
\begin{pro}
Let $N$ be a Nijenhuis operator on a $3$-Hom-pre-Lie algebra  $(A,\{\cdot,\cdot,\cdot\},\alpha)$ and $R:A\to A$ be a Rota-Baxter operator such that $N\circ R=R\circ N$. Then $N$ is a  Nijenhuis operator on  $3$-Hom-L-dendriform algebras $(A,\nw,\ne,\alpha)$ defined in Corollary \ref{3HomLden by invert O-op}.

\end{pro}
\begin{proof}
It can be checked immediately. So we omit details.
\end{proof}
It obvious to show that:
\begin{pro}
Let $(A,\nw,\ne,\alpha)$ be a $3$-Hom-L-dendriform algebra. If an endomorphism $N$ is a derivation, then $N$ is a Nijenhuis operator if and only if  $N$ is a Rota-Baxter operator  of weight $0$ on $A$.
\end{pro}
\subsection{Product and complex structures on $3$-Hom-L-dendriform algebras}
Throughout  this section, we work over the field $\mathbb{R}$ the complex filed $\mathbb{C}$ and all the vector spaces are fined dimension. We  introduce the notion of a product and complex structure on a $3$-Hom-L-dendriform algebra using the Nijenhuis
condition as the integrability condition.

\begin{defi}
  Let $(A,\nw,\ne,\alpha)$ be a $3$-Hom-L-dendriform algebra. An almost product structure on the $3$-Hom-L-dendriform algebra  $(A,\nw,\ne,\alpha)$ is a linear endomorphism $E:A\to A$ satisfying $E^2 = id_A$. An almost
product structure is called a product structure if it is a Nijenhuis operator.
\end{defi}
\begin{rem}
  One can understand a product structure on the $3$-Hom-L-dendriform algebra  $(A,\nw,\ne,\alpha)$ as an linear map  $E:A\to A$ satisfying
  \begin{eqnarray}
  E^2=id_A,&&E\a=\a E,\nonumber\\
E\nw(x,y,z)&=&\nw(Ex,Ey,Ez)+\nw(Ex,y,z)+\nw(x,Ey,z)+\nw(x,y,Ez)\nonumber\\
\label{product structure1}&&-E\nw(Ex,Ey,z)-E\nw(x,Ey,Ez)-E\nw(Ex,y,Ez),\\
E\ne(x,y,z)&=&\ne(Ex,Ey,Ez)+\ne(Ex,y,z)+\ne(x,Ey,z)+\ne(x,y,Ez)\nonumber\\
&&-E\ne(Ex,Ey,z)-E\ne(x,Ey,Ez)-E\ne(Ex,y,Ez)
\label{product structure2}.
\end{eqnarray}

\end{rem}
\begin{thm}\label{product-structure-subalgebra}
Let $(A,\nw,\ne,\alpha)$ be a $3$-Hom-L-dendriform algebra. Then $(A,\nw,\ne,\alpha)$ has a product structure if and only if $ A$ admits a decomposition:
\begin{eqnarray}
 A= A_+\oplus A_-,
\end{eqnarray}
where $ A_+$ and $ A_-$ are subalgebras of $ A$.
\end{thm}
\begin{proof}

Let $E$ be a product structure on $ A$. By $E^2=id_A$ , we have $ A= A_+\oplus A_-$, where $ A_+$ and $ A_-$ are the eigenspaces of $ A$ associated to the eigenvalues $\pm1$. For all $x\in A_+$, we have $\a(x)\in A_+$, in fact
$$E\a(x)=\a E(x)=\a(x).$$
Now, let $x_1,x_2,x_3\in A_+$, we have
\begin{eqnarray*}
E\nw(x_1,x_2,x_3)&=&\nw(Ex_1,Ex_2,Ex_3)+\nw(Ex_1,x_2,x_3)+\nw(x_1,Ex_2,x_3)+\nw(x_1,x_2,Ex_3)\\
&&-E\nw(Ex_1,Ex_2,x_3)-E\nw(x_1,Ex_2,Ex_3)-E\nw(Ex_1,x_2,Ex_3)\\
&=&4\nw(x_1,x_2,x_3)-3E\nw(x_1,x_2,x_3).
\end{eqnarray*}
Thus, we have $\nw(x_1,x_2,x_3)\in A_{+}$. We can also show that $\ne(x_1,x_2,x_3)\in A_{+}$, which implies that $ A_+$ is a subalgebra. Similarly, we can show that $ A_-$ is a subalgebra.

Conversely, we define a linear endomorphism $E: A\to A$ by
\begin{eqnarray}\label{eq:productE}
E(x+a)=x-a,\,\,\,\,\forall x\in A_+,a\in A_-.
\end{eqnarray}
Obviously we have $E^2=id_A$ and $E\a=\a E$.

Since $ A_+$ is a subalgebra of $ A$, for all $x_1,x_2,x_3\in A_+$, we have
\begin{eqnarray*}
&&\nw(Ex_1,Ex_2,Ex_3)+\nw(Ex_1,x_2,x_3)+\nw(x_1,Ex_2,x_3)+\nw(x_1,x_2,Ex_3)\\
&&-E\nw(Ex_1,Ex_2,x_3)-E\nw(x_1,Ex_2,Ex_3)-E\nw(Ex_1,x_2,Ex_3)\\
&=&4\nw(x_1,x_2,x_3)-3E\nw(x_1,x_2,x_3)=\nw(x_1,x_2,x_3)\\
&=&E\nw(x_1,x_2,x_3),
\end{eqnarray*}
which implies that \eqref{product structure1} holds for all $x_1,x_2,x_3\in A_+$ and by the same computation, we can prove the identity \eqref{product structure2}. Similarly, we can show that \eqref{product structure1}-\eqref{product structure2} holds for all $x,y,z\in A$.
Therefore,   $E$ is a product structure on $ A$.  \end{proof}
\begin{pro}
  Let $E$ be an almost product structure on a $3$-Hom-L-dendriform algebra $(A,\nw,\ne,\alpha)$ commuting wilth $\a$.
   \begin{enumerate}
     \item If $E$ satisfies
     \begin{equation}\label{abel-product-0}
     \begin{split}
&E\nw(x,y,z)=\nw(Ex,y,z)=\nw(x,y,Ez),\\
&E\ne(x,y,z)=\ne(Ex,y,z)=\ne(x,Ey,z)=\ne(x,y,Ez),\ \forall \ x,y,z\in A.
\end{split}
\end{equation}
Then $E$ is a product structure on $A$ called a strict product structure.
     \item If $E$ satisfies
     \begin{eqnarray}\label{abel-product}
&&\nw(x,y,z)=-\nw(x,Ey,Ez)-\nw(Ex,y,Ez)-\nw(Ex,Ey,z),\\&&\ne(x,y,z)=-\ne(x,Ey,Ez)-\ne(Ex,y,Ez)-\ne(Ex,Ey,z),\ \forall \ x,y,z\in A.
\end{eqnarray}
Then $E$ is a product structure on $A$ called an abilian  product structure.
 \item If $E$ satisfies
     \begin{eqnarray}
&&\nw(x,y,z)=E(\nw(x,y,Ez)+\nw(Ex,y,z)+\nw(x,Ey,z)),\\&&\ne(x,y,z)=E(\ne(x,y,Ez)+\ne(Ex,y,z)+\ne(x,Ey,z)),\ \forall \ x,y,z\in A.
\end{eqnarray}
Then $E$ is a product structure on $A$ called an strong  product structure.
\item If $E$ satisfies
     \begin{eqnarray}\label{abel-product-0}
E\nw(x,y,z)=\nw(Ex,Ey,Ez),\ E\ne(x,y,z)=\ne(Ex,Ey,Ez),\ \forall \ x,y,z\in A.
\end{eqnarray}
Then $E$ is a product structure on $A$ called a perfect product structure.

   \end{enumerate}\end{pro}
   \begin{proof}
     \begin{enumerate}
       \item By \eqref{abel-product-0} and $E^2=id_A$, we have
\begin{eqnarray*}
&&\nw(Ex,Ey,Ez)+\nw(Ex,y,z)+\nw(x,Ey,z)+\nw(x,y,Ez)\\
&&-E\nw(Ex,Ey,z)-E\nw(x,Ey,Ez)-E\nw(Ex,y,Ez)\\
&=&\nw(Ex,Ey,Ez)+E\nw(x,y,z)+\nw(x,Ey,z)+\nw(x,y,Ez)\\
&&-\nw(E^2x,Ey,z)-\nw(Ex,Ey,Ez)-\nw(E^2x,y,Ez)\\
&=&E\nw(x,y,z).
\end{eqnarray*}
Similarly, we can prove the identity \eqref{product structure2}.
Thus, $E$ is a product structure on $A$.
       \item By \eqref{abel-product} and $E^2=id_A$, we have
\begin{eqnarray*}
&&\nw(Ex,Ey,Ez)+\nw(Ex,y,z)+\nw(x,Ey,z)+\nw(x,y,Ez)\\
&&-E\nw(Ex,Ey,z)-E\nw(x,Ey,Ez)-E\nw(Ex,y,Ez)\\
&=&-\nw(Ex,E^2y,E^2z)-\nw(E^2x,Ey,E^2z)-\nw(E^2x,E^2y,Ez)\\
&&+\nw(Ex,y,z)+\nw(x,Ey,z)+\nw(x,y,Ez)+E\nw(x,y,z)\\
&=&E\nw(x,y,z).
\end{eqnarray*}
Thus,   $E$ is a product structure on $A$.\\
Items 3. and 4.  can be proved  similarly.
     \end{enumerate}
   \end{proof}
   \begin{rem}
     A strict product structure on a $3$-Hom-L-dendriform algebra    is a perfect product structure.
   \end{rem}

\begin{defi}\label{complex}
  Let $(A,\nw,\ne,\alpha)$ be a $3$-Hom-L-dendriform algebra. An almost complex  structure on the $3$-Hom-L-dendriform algebra  $(A,\nw,\ne,\alpha)$ is a linear endomorphism $J:A\to A$ satisfying $J^2 = -id_A$. An almost
 structure is called a complex  structure if it is a Nijenhuis operator.
\end{defi}
\begin{rem}\label{complex1}
  One can understand a complex  structure on the $3$-Hom-L-dendriform algebra  $(A,\nw,\ne,\alpha)$ as an linear map  $J:A\to A$ satisfying
  \begin{eqnarray}
  J^2=-id_A,&&J\a=\a J,\nonumber\\
J\nw(x,y,z)&=&-\nw(Jx,Jy,Jz)+\nw(Jx,y,z)+\nw(x,Jy,z)+\nw(x,y,Jz)\nonumber\\
\label{product complex structure1}&&+J\nw(Jx,Jy,z)+J\nw(x,Jy,Jz)+J\nw(Jx,y,Jz),\\
J\ne(x,y,z)&=&-\ne(Jx,Jy,Jz)+\ne(Jx,y,z)+\ne(x,Jy,z)+\ne(x,y,Jz)\nonumber\\
&&+J\ne(Jx,Jy,z)+J\ne(x,Jy,Jz)+J\ne(Jx,y,Jz)
\label{product complex structure2}.
\end{eqnarray}

\end{rem}

\begin{rem}
One can also use definition \ref{complex} to define the notion of a complex structure on a
complex $3$-Hom-L-dendriform algebra, considering $J$ to be $\mathbb C$-linear. However, this is not very interesting since
for a complex $3$-Hom-L-dendriform algebra, there is a one-to-one correspondence between such $\mathbb C$-linear complex
structures and product structures.
\end{rem}

Consider  $A_{\mathbb C}=A\otimes_{\mathbb R} \mathbb C\cong\{x+iy|x,y\in A\}$, the complexification of the
real $3$-Hom-L-dendriform algebra $A$. We will denote it by $(A_{\mathbb C},\nw_{A_{\mathbb C}},\ne_{A_{\mathbb C}},\a_{A_{\mathbb C}})$. We have an equivalent description of the integrability condition given in
Remark \ref{complex1}. For any $z=x+iy\in A_{\mathbb C}$ we denote its conjugate by $\overline{z}=x-iy,\,\,x,y\in A$. Then the conjugation in $ A_{\mathbb C}$ is a complex antilinear, involutive automorphism of the complex vector space $A_{\mathbb C}$.

\begin{thm}\label{product-structure-subalgebra}
Let $(A,\nw,\ne,\alpha)$ be a real $3$-Hom-L-dendriform algebra. Then $(A,\nw,\ne,\alpha)$ has a complex product structure if and only if $A_{\mathbb C}$ admits a decomposition:
\begin{eqnarray}
A_{\mathbb C}=\mathfrak{q}\oplus\overline{\mathfrak{q}},
\end{eqnarray}
where $\mathfrak{q}$ is a complex subalgebras of $A_{\mathbb C}$.\end{thm}
\begin{proof}
We  extend the complex structure $J$ $\mathbb C$-linearly, which is denoted by $J_{\mathbb C}$, i.e. $J_{\mathbb C}: A_{\mathbb C}\longrightarrow  A_{\mathbb C}$ is defined as
\begin{equation}\label{eq:JC}
J_{\mathbb C}(x+iy)=Jx+iJy,\quad \forall x,y\in A.
\end{equation} Then $J_{\mathbb C}$ is a $\mathbb C$-linear endomorphism on $ A_{\mathbb C}$ satisfying $J_{\mathbb C}^2=-id_{A_{\mathbb C}}$, $J_{\mathbb C}\a_{\mathbb C}=\a_{\mathbb C}J_{\mathbb C}$ and the integrability conditions \eqref{product complex structure1} and \eqref{product complex structure2} on $ A_{\mathbb C}$. Denote by $ A_{\pm i}$ the corresponding eigenspaces of $ A_{\mathbb C}$ associated to the eigenvalues $\pm i$ and there holds:
   \begin{eqnarray*}
 A_{\mathbb C}= A_{i}\oplus A_{-i}.
\end{eqnarray*}
  It is straightforward to see that  $ A_{i}=\{x-iJx|x\in A\}$ and $ A_{-i}=\{x+iJx|x\in A\}$. Therefore, we have $ A_{-i}= \overline{A_{i}}$. It is easy to check that $\a_{\mathbb C}(A_i)\subset A_i$.
For all $X,Y,Z\in A_{i}$, we have
\begin{eqnarray*}
J_{\mathbb C}\nw_{ A_{\mathbb C}}(X,Y,Z)&=&-\nw_{ A_{\mathbb C}}(J_{\mathbb C}X,J_{\mathbb C}Y,J_{\mathbb C}Z)+\nw_{ A_{\mathbb C}}(J_{\mathbb C}X,Y,Z)+\nw_{ A_{\mathbb C}}(X,J_{\mathbb C}Y,Z)+\nw_{ A_{\mathbb C}}(X,Y,J_{\mathbb C}Z)\\
&&+J_{\mathbb C}\nw_{ A_{\mathbb C}}(J_{\mathbb C}X,J_{\mathbb C}Y,Z)+J_{\mathbb C}\nw_{ A_{\mathbb C}}(X,J_{\mathbb C}Y,J_{\mathbb C}Z)+J_{\mathbb C}\nw_{ A_{\mathbb C}}(J_{\mathbb C}X,Y,J_{\mathbb C}Z)\\
&=&4i\nw_{ A_{\mathbb C}}(X,Y,Z)-3J_{\mathbb C}\nw_{ A_{\mathbb C}}(X,Y,Z).
\end{eqnarray*}
Thus, we have $\nw_{ A_{\mathbb C}}(X,Y,Z)\in A_{i}$. Similarly, we can show that $\ne_{ A_{\mathbb C}}(X,Y,Z)\in A_{i}$, which implies that $ A_i$ is a subalgebra of $A_{\mathbb C}$.
% For all $X,Y,Z\in A_{-i}$, we have
%\begin{eqnarray*}
%J_{\mathbb C}[X,Y,Z]_{ A_{\mathbb C}}&=&-[J_{\mathbb C}X,J_{\mathbb C}Y,J_{\mathbb C}Z]_{ A_{\mathbb C}}+[J_{\mathbb C}X,Y,Z]_{ A_{\mathbb C}}+[X,J_{\mathbb C}Y,Z]_{ A_{\mathbb C}}+[X,Y,J_{\mathbb C}Z]_{ A_{\mathbb C}}\\
%&&+J_{\mathbb C}[J_{\mathbb C}X,J_{\mathbb C}Y,Z]_{ A_{\mathbb C}}+J_{\mathbb C}[X,J_{\mathbb C}Y,JZ]_{ A_{\mathbb C}}+J_{\mathbb C}[J_{\mathbb C}X,Y,J_{\mathbb C}Z]_{ A_{\mathbb C}}\\
%&=&-4i[X,Y,Z]_{ A_{\mathbb C}}-3J_{\mathbb C}[X,Y,Z]_{ A_{\mathbb C}}.
%\end{eqnarray*}Thus, we have $[X,Y,Z]_{ A_{\mathbb C}}\in A_{-i}$.

Conversely, we define a complex linear endomorphism $J_{\mathbb C}: A_{\mathbb C}\to A_{\mathbb C}$ by
\begin{eqnarray}\label{defi-complex-structure}
J_{\mathbb C}(X+\overline{Y})=iX-i\overline{Y},\,\,\,\,\forall X,Y\in\mathfrak{q}.
\end{eqnarray}
Since the conjugation  is a $\mathbb C$-antilinear, involutive automorphism of $ A_{\mathbb C}$, we have
\begin{eqnarray*}
J_{\mathbb C}^2(X+\overline{Y})=J_{\mathbb C}(iX-i\overline{Y})=J_{\mathbb C}(iX+\overline{iY})=i(iX)-i\overline{iY}=-X-\overline{Y},
\end{eqnarray*}
then $J_{\mathbb C}^2=-id$. Moreover, for any $X,Y\in\mathfrak{q}$, we have
\begin{align*}
J_{\mathbb C}\a_{\mathbb C}(X+\overline{Y})&=J_{\mathbb C}(\a_{\mathbb C}(X)+\a_{\mathbb C}(\overline{Y}))=J_{\mathbb C}(\a_{\mathbb C}(X)+\overline{\a_{\mathbb C}(Y)})\\&=i\a_{\mathbb C}(X)-i\overline{\a_{\mathbb C}(Y)}=\a_{\mathbb C}(i X)-\a_{\mathbb C}(i\overline{Y})=\a_{\mathbb C}J_{\mathbb C}(X+\overline{Y}).\end{align*}
Since $\mathfrak{q}$ is a subalgebra of $ A_{\mathbb C}$ then,  for all $X,Y,Z\in\mathfrak{q}$, we have
\begin{eqnarray*}
&&-\nw_{ A_{\mathbb C}}(J_{\mathbb C}X,J_{\mathbb C}Y,J_{\mathbb C}Z)+\nw_{ A_{\mathbb C}}(J_{\mathbb C}X,Y,Z)+\nw_{ A_{\mathbb C}}(X,J_{\mathbb C}Y,Z)+\nw_{ A_{\mathbb C}}(X,Y,J_{\mathbb C}Z)\\
&&+J_{\mathbb C}\nw_{ A_{\mathbb C}}(J_{\mathbb C}X,J_{\mathbb C}Y,Z)+J_{\mathbb C}\nw_{ A_{\mathbb C}}(X,J_{\mathbb C}Y,J_{\mathbb C}Z)+J_{\mathbb C}\nw_{ A_{\mathbb C}}(J_{\mathbb C}X,Y,J_{\mathbb C}Z)\\
&=&4i\nw_{ A_{\mathbb C}}(X,Y,Z)-3J_{\mathbb C}\nw_{ A_{\mathbb C}}(X,Y,Z)=i\nw_{ A_{\mathbb C}}(X,Y,Z)\\
&=&J_{\mathbb C}\nw_{ A_{\mathbb C}}(X,Y,Z),
\end{eqnarray*}
which implies that $J_{\mathbb C}$ satisfies   \eqref{product complex structure1} for all $X,Y,Z\in\mathfrak{q}$. Using a similar computation, we can check that $J_{\mathbb C}$ satisfies   \eqref{product complex structure2} for any $X,Y,Z\in\mathfrak{q}$.  Similarly, we can show that $J_{\mathbb C}$ satisfies   \eqref{product complex structure1} and   \eqref{product complex structure2}, for all $\mathcal{X},\mathcal{Y},\mathcal{Z}\in A_{\mathbb C}$. Let $J\in gl( A)$ given by $$J\triangleq J_{\mathbb C}|_{ A}.$$
$J$ is well-defined.  Since $J_{\mathbb C}$  satisfies \eqref{product complex structure1}, \eqref{product complex structure2}, $J_{\mathbb C}^2=-id_{\mathbb C}$ and  $\a_{\mathbb C}J_{\mathbb C}=J_{\mathbb C}\a_{\mathbb C}$ on $ A_{\mathbb C}$ then $J$ is a complex structure on $ A$.
 \end{proof}
\begin{lem}
  Let $J$ be an almost complex structure on a real $3$-Hom-L-dendriform algebra $(A,\nw,\ne,\alpha)$ commuting with $\a$.
    If $J$ satisfies
     \begin{eqnarray}\label{abel-complex-0}
&&J\nw(x,y,z)=\nw(Jx,y,z)=\nw(x,y,Jz),\\&&\label{abel-complex-01} J\ne(x,y,z)=\ne(Jx,y,z)=\ne(x,Jy,z)=\ne(x,y,Jz),\ \forall \ x,y,z\in A.
\end{eqnarray}
Then $J$ is a complex structure on $A$ called a strict complex structure.
\end{lem}
\begin{proof}
  Using identity \eqref{abel-complex-01} and $J^2=-id_A$, we obtain
\begin{eqnarray*}
&&-\ne(Jx,Jy,Jz)+\ne(Jx,y,z)+\ne(x,Jy,z)+\ne(x,y,Jz)\\
&&+J\ne(Jx,Jy,z)+J\ne(x,Jy,Jz)+J\ne(Jx,y,Jz)\\
&=&-\ne(Jx,Jy,Jz)+J\ne(x,y,z)+\ne(x,Jy,z)+\ne(x,y,Jz)\\
&&+\ne(J^2x,Jy,z)+\ne(Jx,Jy,Jz)+\ne(J^2x,y,Jz)\\
&=&J\ne(x,y,z).
\end{eqnarray*}
Similarly, we can get \eqref{abel-complex-0}.
Thus, we obtain that $J$ is a complex structure on $A$.
\end{proof}
Let $J$ be a complex structure on $A$. Define   two new products $\nw_J,\ne_J:\otimes^3A\to A$ by
\begin{eqnarray}\label{J-bracket1}
&&\nw_J(x,y,z)\triangleq \frac{1}{4}(\nw(x,y,z)-\nw(x,Jy,Jz)-\nw(Jx,y,Jz)-\nw(Jx,Jy,z)),\\\label{J-bracket2}
&&\ne_J(x,y,z)\triangleq \frac{1}{4}(\ne(x,y,z)-\ne(x,Jy,Jz)-\ne(Jx,y,Jz)-\ne(Jx,Jy,z)),\,\,\,\,\forall x,y,z\in A.
\end{eqnarray}

\begin{pro}\label{subalgebra-iso}
Let $J$ be a complex structure on a real $3$-Hom-L-dendriform algebra $(A,\nw,\ne,\alpha)$ . Then the tuple $(A,\nw_J,\ne_J,\alpha)$ is a real $3$-Hom-L-dendriform algebra. Moreover, $J$ is a strict complex structure on  $(A,\nw_J,\ne_J,\alpha)$ and the corresponding complex $3$-Hom-L-dendriform algebra $(A,\nw_J,\ne_J,\alpha)$  is isomorphic to the complex $3$-Hom-L-dendriform algebra $\mathfrak{q}$ define in Theorem \ref{product-structure-subalgebra}.
\end{pro}
\begin{proof}
We will first prove that   $(A,\nw_J,\ne_J,\alpha)$ is a real $3$-Hom-L-dendriform algebra. To minimize the computations, we use a new approach to do this.
Since $J$ is  complex structure on $A$. We can define  two map $\varphi:A\to\mathfrak{q}$  as following:
\begin{eqnarray*}
\varphi(x)=\frac{1}{2}(x-iJx).
\end{eqnarray*}
It is straightforward to deduce that $\varphi$ is complex linear isomorphism between complex vector spaces.
%
%Since $J$ be a strict complex structure on a real $3$-Hom-L-dendriform algebra $(A,\nw_J,\ne_J,\alpha)$. Then with the complex vector space structure  defined above,  $(\g,[\cdot,\cdot,\cdot]_\g)$ is a complex $3$-Lie algebra.
% In fact,  the fact that the $3$-Lie bracket is complex trilinear follows from
%\begin{eqnarray*}
%[(a+bi)x,y,z]_\g&=&[ax+bJx,y,z]_\g=a[x,y,z]_\g+b[Jx,y,z]_\g\\
%                &=&a[x,y,z]_\g+bJ[x,y,z]_\g=(a+bi)[x,y,z]_\g
%\end{eqnarray*}
%using \eqref{adapt} and \eqref{complex-space}. \vspace{2mm}

  Since $J\a=\a J $, then $\varphi\a=\a \varphi $. By \eqref{product complex structure1}, for all $x,y,z\in A$, we have
\begin{eqnarray}
\nonumber\nw(\varphi(x),\varphi(y),\varphi(z))&=&\frac{1}{8}\nw(x-iJx,y-iJy,z-iJz)\\
                                                  \nonumber &=&\frac{1}{8}(\nw(x,y,z)-\nw(x,Jy,Jz)-\nw(Jx,y,Jz)-\nw(Jx,Jy,z))\\
                                                   \nonumber&&-\frac{1}{8}i(\nw(x,y,Jz)+\nw(x,Jy,z)+\nw(Jx,y,z)-\nw(Jx,Jy,Jz))\\
                                                  \nonumber &=&\frac{1}{8}(\nw(x,y,z)-\nw(x,Jy,Jz)-\nw(Jx,y,Jz)-\nw(Jx,Jy,z))\\
                                                  \nonumber &&-\frac{1}{8}iJ(\nw(x,y,z)-\nw(x,Jy,Jz)-\nw(Jx,y,Jz)-\nw(Jx,Jy,z))\\
                                                   \label{eq:Jiso}&=&\varphi\nw_J(x,y,z).
\end{eqnarray}
Thus, we have $\nw_J(x,y,z)=\varphi^{-1}(\nw(\varphi(x),\varphi(y),\varphi(z))$. Similarly, we get $\ne_J(x,y,z)=\varphi^{-1}\ne(\varphi(x),\varphi(y),\varphi(z))$. Since $J $ is a complex structure, $\mathfrak{q}$ is a $3$-Hom-L-dendriform subalgebra. Hence $(A,\nw_J,\ne_J,\alpha)$ is a real $3$-Hom-L-dendriform algebra.

By \eqref{abel-complex-01}, for all $x,y,z\in A$, we have
\begin{eqnarray*}
J\nw_J(x,y,z)&=&\frac{1}{4}J(\nw(x,y,z)-\nw_J(x,Jy,Jz)-\nw_J(Jx,y,Jz)-\nw_J(Jx,Jy,z))\\
          &=&\frac{1}{4}(-\nw_J(Jx,Jy,Jz)+\nw_J(Jx,y,z)+\nw_J(x,Jy,z)+\nw_J(x,y,Jz))\\&=&\frac{1}{4}(-\nw_J(Jx,Jy,Jz)+\nw_J(Jx,y,z)-\nw_J(J^2x,Jy,z)-\nw_J(J^2x,y,Jz))\\
          &=&\nw_J(Jx,y,z),
\end{eqnarray*} Similarly, we can check that $$J\nw(x,y,z)=\nw(x,y,Jz)\ \textrm{and}\   J\ne(x,y,z)=\ne(Jx,y,z)=\ne(x,Jy,z)=\ne(x,y,Jz),\ \forall \ x,y,z\in A.$$
which implies that $J$ is a strict complex structure on  $(A,\nw_J,\ne_J,\alpha)$.
Identity \eqref{eq:Jiso} means that$(A,\nw_J,\ne_J,\alpha)$  is isomorphic to the complex $3$-Hom-L-dendriform algebra $\mathfrak{q}$.

\end{proof}

\begin{pro}
Let $J$ be a complex structure  on a real $3$-Hom-L-dendriform algebra $(A,\nw,\ne,\alpha)$. Then $J$ is a strict complex structure on   $(A,\nw,\ne,\alpha)$  if and only if $\nw_J=\nw$ and $\ne_J=\ne$.
\end{pro}
\begin{proof} If $J$ is a strict complex structure on $(A,\nw,\ne,\alpha)$, by $J\nw(x,y,z)=\nw(Jx,y,z)=\nw(x,y,Jz)$, we have
\begin{eqnarray*}
\nw_J(x,y,z)=\frac{1}{4}(\nw(x,y,z)-\nw(x,Jy,Jz)-\nw(Jx,y,Jz)-\nw(Jx,Jy,z))=\nw(x,y,z).
\end{eqnarray*}
Similarly, we can check that $\ne_J(x,y,z)=\ne(x,y,z)$.

Conversely, if $\nw_J=\nw$ and $\ne_J=\ne$, we have
$$-3\nw(x,y,z)=\nw(x,Jy,Jz)+\nw(Jx,y,Jz)+\nw(Jx,Jy,z).$$
Then by  condition \eqref{product complex structure1}, we obtain
\begin{eqnarray*}
4J\nw_J(x,y,z)&=&-\nw(Jx,Jy,Jz)+\nw(Jx,y,z)+\nw(x,Jy,z)+\nw(x,y,Jz)\\
            &=&3\nw(Jx,y,z)+\nw(Jx,y,z)\\
            &=&4\nw(Jx,y,z),
\end{eqnarray*}
which implies that $J\nw(x,y,z)=\nw(Jx,y,z)$. Similarly, using \eqref{product complex structure2}, we obtain $J\ne(x,y,z)=\ne(Jx,y,z)$. The proof is finished.\end{proof}
\begin{lem}Let $J$ be an almost complex structure on a real $3$-Hom-L-dendriform algebra $(A,\nw,\ne,\alpha)$ commuting with $\a$.
      If $J$ satisfies
     \begin{eqnarray}\label{AbelianComplexStructure1}
&&\nw(x,y,z)=\nw(x,Jy,Jz)+\nw(Jx,y,Jz)+\nw(Jx,Jy,z),\\\label{AbelianComplexStructure2}&&\ne(x,y,z)=\ne(x,Jy,Jz)+\ne(Jx,y,Jz)+\ne(Jx,Jy,z),\ \forall \ x,y,z\in A.
\end{eqnarray}
Then $J$ is a complex structure on $A$ called an abelian  complex structure.
 \end{lem}
 \begin{proof}
   By \eqref{AbelianComplexStructure1} and $J^2=-id_A$, we have
\begin{eqnarray*}
&&-\nw(Jx,Jy,Jz)+\nw(Jx,y,z)+\nw(x,Jy,z)+\nw(x,y,Jz)\\
&&+J\nw(Jx,Jy,z)+J\nw(x,Jy,Jz)+J\nw(Jx,y,Jz)\\
&=&-\nw(Jx,J^2y,J^2z)-\nw(J^2x,Jy,J^2z)-\nw(J^2x,J^2y,Jz)\\
&&+\nw(Jx,y,z)+\nw(x,Jy,z)+\nw(x,y,Jz)+J\nw(x,y,z)\\
&=&J\nw(x,y,z).
\end{eqnarray*}
Using  \eqref{AbelianComplexStructure2}, we can prove the condition \eqref{product complex structure2}. Thus, we obtain that $J$ is a complex structure on $A$.
 \end{proof}
 \begin{lem}Let $J$ be an almost complex structure on a real $3$-Hom-L-dendriform algebra $(A,\nw,\ne,\alpha)$ commuting with $\a$. If $J$ satisfies
     \begin{eqnarray}\label{StrongComplexStructure1}
&&\nw(x,y,z)=-J(\nw(x,y,Jz)+\nw(Jx,y,z)+\nw(x,Jy,z)),\\\label{StrongComplexStructure2}&&\ne(x,y,z)=-J(\ne(x,y,Jz)+\ne(Jx,y,z)+\ne(x,Jy,z)),\ \forall \ x,y,z\in A.
\end{eqnarray}
Then $J$ is a complex structure on $A$ called an strong  complex structure.
\end{lem}
\begin{proof}
  By \eqref{StrongComplexStructure1} and $J^2=-id_A$, we have
\begin{eqnarray*}
&&-\nw(Jx,Jy,Jz)+\nw(Jx,y,z)+\nw(x,Jy,z)+\nw(x,y,Jz)\\
&&+J\nw(Jx,Jy,z)+J\nw(x,Jy,Jz)+J\nw(Jx,y,Jz)\\
&=&J\nw(J^2x,Jy,Jz)+J\nw(Jx,J^2y,Jz)+J\nw(Jx,Jy,J^2z)+J\nw(x,y,z)\\
&&+J\nw(Jx,Jy,z)+J\nw(x,Jy,Jz)+J\nw(Jx,y,Jz)\\
&=&J\nw(x,y,z).
\end{eqnarray*}
Using  \eqref{StrongComplexStructure2}, we can prove the condition \eqref{product complex structure2}.
Thus, $J$ is a complex structure on $A$.
\end{proof}
\begin{lem} Let $J$ be an almost complex structure on a real $3$-Hom-L-dendriform algebra $(A,\nw,\ne,\alpha)$ commuting wilth $\a$. If $J$ satisfies
     \begin{eqnarray}\label{PerfectComplexStructure}
J\nw(x,y,z)=-\nw(Jx,Jy,Jz),\ J\ne(x,y,z)=-\ne(Jx,Jy,Jz),\ \forall \ x,y,z\in A.
\end{eqnarray}
Then $E$ is a complex structure on $A$ called a perfect complex structure.
   \end{lem}

\begin{proof}
  By \eqref{PerfectComplexStructure} and $J^2=-id_A$, we have
\begin{eqnarray*}
&&-\nw(Jx,Jy,Jz)+\nw(Jx,y,z)+\nw(x,Jy,z)+\nw(x,y,Jz)\\
&&+J\nw(Jx,Jy,z)+J\nw(x,Jy,Jz)+J\nw(Jx,y,Jz)\\
&=&J\nw(x,y,z)+\nw(Jx,y,z)+\nw(x,Jy,z)+\nw(x,y,Jz)\\
&&-\nw(J^2x,J^2y,Jz)-\nw(Jx,J^2y,J^2z)-\nw(J^2x,Jy,J^2z)\\
&=&J\nw(x,y,z)
\end{eqnarray*} and \begin{eqnarray*}
&&-\ne(Jx,Jy,Jz)+\ne(Jx,y,z)+\ne(x,Jy,z)+\ne(x,y,Jz)\\
&&+J\ne(Jx,Jy,z)+J\ne(x,Jy,Jz)+J\ne(Jx,y,Jz)\\
&=&J\ne(x,y,z)+\ne(Jx,y,z)+\ne(x,Jy,z)+\ne(x,y,Jz)\\
&&-\ne(J^2x,J^2y,Jz)-\ne(Jx,J^2y,J^2z)-\ne(J^2x,Jy,J^2z)\\
&=&J\ne(x,y,z).
\end{eqnarray*}
Thus, $J$ is a complex structure on $A$.
\end{proof}
The following result  illustrates the relation between a complex structure and a product structure on a complex $3$-Hom-L-dendriform algebra.

\begin{pro}\label{equivalent}
Let $(A,\nw,\ne,\alpha)$ be a complex $3$-Hom-L-dendriform algebra. Then $E$ is a product structure on $A$ if and only if $J=iE$ is a complex structure on $A$.
\end{pro}

\begin{proof} Let $E$ be a product structure on $A$. We have $J^2=i^2E^2=-id_A.$ Thus, $J$ is an almost complex structure on $A$. Since $E$ satisfies the  condition \eqref{product structure1}, we have
\begin{eqnarray*}
J\nw(x,y,z)&=&iE\nw(x,y,z)\\
                 &=&-\nw(iEx,iEy,iEz)+\nw(iEx,y,z)+\nw(x,iEy,z)+\nw(x,y,iEz)\\
                 &&+iE\nw(iEx,iEy,z)+iE\nw(x,iEy,iEz)+iE\nw(iEx,y,iEz)\\
                 &=&-\nw(Jx,Jy,Jz)+\nw(Jx,y,z)+\nw(x,Jy,z)+\nw(x,y,Jz)\\
                 &&+J\nw(Jx,Jy,z)+J\nw(x,Jy,Jz)+J\nw(Jx,y,Jz).
\end{eqnarray*}
Similarly, using  \eqref{product structure2} we can check that
\begin{eqnarray*}
J\ne(x,y,z)        &=&-\ne(Jx,Jy,Jz)+\ne(Jx,y,z)+\ne(x,Jy,z)+\ne(x,y,Jz)\\
                 &&+J\ne(Jx,Jy,z)+J\ne(x,Jy,Jz)+J\ne(Jx,y,Jz).
\end{eqnarray*} Thus, $J$ is a complex structure on the complex $3$-Hom-L-dendriform algebra $A$.

The converse part can be proved similarly and we omit details.\end{proof}

%%%%%%%%%%%%%%%%%%%%%

\end{document}